\documentclass[10pt,reqno]{amsart}
\usepackage{latexsym,amsmath,amssymb,amscd}
\usepackage[all]{xy}

\def\today{\ifcase \month \or
   January \or February \or March \or April \or
   May \or June \or July \or August \or
   September \or October \or November \or December \fi
   \space\number\day , \number\year}
   
\makeatletter
  \newcommand\@dotsep{4.5}
  \def\@tocline#1#2#3#4#5#6#7{\relax
     \ifnum #1>\c@tocdepth 
     \else
     \par \addpenalty\@secpenalty\addvspace{#2}%
     \begingroup \hyphenpenalty\@M
     \@ifempty{#4}{%
     \@tempdima\csname r@tocindent\number#1\endcsname\relax
        }{%
         \@tempdima#4\relax
           }%
      \parindent\z@ \leftskip#3\relax \advance\leftskip\@tempdima\relax
      \rightskip\@pnumwidth plus1em \parfillskip-\@pnumwidth
       #5\leavevmode\hskip-\@tempdima #6\relax
       \leaders\hbox{$\m@th
       \mkern \@dotsep mu\hbox{.}\mkern \@dotsep mu$}\hfill
       \hbox to\@pnumwidth{\@tocpagenum{#7}}\par
       \nobreak
        \endgroup
         \fi}
\makeatother 

\begin{document}

\makeatletter
\@addtoreset{figure}{section}
\def\thefigure{\thesection.\@arabic\c@figure}
\def\fps@figure{h,t}
\@addtoreset{table}{bsection}

\def\thetable{\thesection.\@arabic\c@table}
\def\fps@table{h, t}
\@addtoreset{equation}{section}
\def\theequation{\arabic{equation}}
\makeatother

\newcommand{\bfi}{\bfseries\itshape}

\newtheorem{theorem}{Theorem}
\newtheorem{corollary}[theorem]{Corollary}
\newtheorem{definition}[theorem]{Definition}
\newtheorem{example}[theorem]{Example}
\newtheorem{lemma}[theorem]{Lemma}
\newtheorem{notation}[theorem]{Notation}
\newtheorem{proposition}[theorem]{Proposition}
\newtheorem{remark}[theorem]{Remark}
\newtheorem{setting}[theorem]{Setting}

\numberwithin{theorem}{section}
\numberwithin{equation}{section}

\renewcommand{\1}{{\bf 1}}
\newcommand{\Ad}{{\rm Ad}}
\newcommand{\Alg}{{\rm Alg}\,}
\newcommand{\Aut}{{\rm Aut}\,}
\newcommand{\ad}{{\rm ad}}
\newcommand{\Borel}{{\rm Borel}}
\newcommand{\botimes}{\bar{\otimes}}
\newcommand{\Cb}{{\mathcal C}_{\rm b}}
\newcommand{\Ci}{{\mathcal C}^\infty}
\newcommand{\Cpol}{{\mathcal C}^\infty_{\rm pol}}
\newcommand{\Der}{{\rm Der}\,}
\newcommand{\de}{{\rm d}}
\newcommand{\ee}{{\rm e}}
\newcommand{\End}{{\rm End}\,}
\newcommand{\ev}{{\rm ev}}
\newcommand{\hotimes}{\widehat{\otimes}}
\newcommand{\id}{{\rm id}}
\newcommand{\ie}{{\rm i}}
\newcommand{\iotaR}{\iota^{\rm R}}
\newcommand{\GL}{{\rm GL}}
\newcommand{\gl}{{{\mathfrak g}{\mathfrak l}}}
\newcommand{\Hom}{{\rm Hom}\,}
\newcommand{\Img}{{\rm Im}\,}
\newcommand{\Ind}{{\rm Ind}}
\newcommand{\Ker}{{\rm Ker}\,}
\newcommand{\Lie}{\text{\bf L}}
\newcommand{\m}{\text{\bf m}}
\newcommand{\pr}{{\rm pr}}
\newcommand{\Ran}{{\rm Ran}\,}
\renewcommand{\Re}{{\rm Re}\,}
\newcommand{\so}{\text{so}}
\newcommand{\spa}{{\rm span}}
\newcommand{\Tr}{{\rm Tr}\,}
\newcommand{\Op}{{\rm Op}}
\newcommand{\U}{{\rm U}}

\newcommand{\CC}{{\mathbb C}}
\newcommand{\RR}{{\mathbb R}}
\newcommand{\TT}{{\mathbb T}}

\newcommand{\Ac}{{\mathcal A}}
\newcommand{\Bc}{{\mathcal B}}
\newcommand{\Cc}{{\mathcal C}}
\newcommand{\Dc}{{\mathcal D}}
\newcommand{\Ec}{{\mathcal E}}
\newcommand{\Fc}{{\mathcal F}}
\newcommand{\Hc}{{\mathcal H}}
\newcommand{\Jc}{{\mathcal J}}
\newcommand{\Lc}{{\mathcal L}}
\renewcommand{\Mc}{{\mathcal M}}
\newcommand{\Nc}{{\mathcal N}}
\newcommand{\Oc}{{\mathcal O}}
\newcommand{\Pc}{{\mathcal P}}
\newcommand{\Rc}{{\mathcal R}}
\newcommand{\Sc}{{\mathcal S}}
\newcommand{\Tc}{{\mathcal T}}
\newcommand{\Vc}{{\mathcal V}}
\newcommand{\Uc}{{\mathcal U}}
\newcommand{\Xc}{{\mathcal X}}
\newcommand{\Yc}{{\mathcal Y}}
\newcommand{\Wig}{{\mathcal W}}

\newcommand{\Bg}{{\mathfrak B}}
\newcommand{\Fg}{{\mathfrak F}}
\newcommand{\Gg}{{\mathfrak G}}
\newcommand{\Ig}{{\mathfrak I}}
\newcommand{\Jg}{{\mathfrak J}}
\newcommand{\Lg}{{\mathfrak L}}
\newcommand{\Pg}{{\mathfrak P}}
\newcommand{\Sg}{{\mathfrak S}}
\newcommand{\Xg}{{\mathfrak X}}
\newcommand{\Yg}{{\mathfrak Y}}
\newcommand{\Zg}{{\mathfrak Z}}

\newcommand{\ag}{{\mathfrak a}}
\newcommand{\bg}{{\mathfrak b}}
\newcommand{\dg}{{\mathfrak d}}
\renewcommand{\gg}{{\mathfrak g}}
\newcommand{\hg}{{\mathfrak h}}
\newcommand{\kg}{{\mathfrak k}}
\newcommand{\mg}{{\mathfrak m}}
\newcommand{\n}{{\mathfrak n}}
\newcommand{\og}{{\mathfrak o}}
\newcommand{\pg}{{\mathfrak p}}
\newcommand{\sg}{{\mathfrak s}}
\newcommand{\tg}{{\mathfrak t}}
\newcommand{\ug}{{\mathfrak u}}
\newcommand{\zg}{{\mathfrak z}}

\newcommand{\ZZ}{\mathbb Z}
\newcommand{\NN}{\mathbb N}
\newcommand{\BB}{\mathbb B}

\newcommand{\ep}{\varepsilon}

\newcommand{\hake}[1]{\langle #1 \rangle }

\newcommand{\scalar}[2]{\langle #1 ,#2 \rangle }
\newcommand{\vect}[2]{(#1_1 ,\ldots ,#1_{#2})}
\newcommand{\norm}[1]{\Vert #1 \Vert }
\newcommand{\normrum}[2]{{\norm {#1}}_{#2}}

\newcommand{\upp}[1]{^{(#1)}}
\newcommand{\p}{\partial}

\newcommand{\opn}{\operatorname}
\newcommand{\slim}{\operatornamewithlimits{s-lim\,}}
\newcommand{\sgn}{\operatorname{sgn}}

\newcommand{\seq}[2]{#1_1 ,\dots ,#1_{#2} }
\newcommand{\loc}{_{\opn{loc}}}

\makeatletter
\title[Continuity of  magnetic Weyl calculus]{Continuity of magnetic Weyl calculus}
\author{Ingrid Belti\c t\u a 
and Daniel Belti\c t\u a
}
\address{Institute of Mathematics ``Simion Stoilow'' 
of the Romanian Academy, 
P.O. Box 1-764, Bucharest, Romania}
\email{Ingrid.Beltita@imar.ro}
\email{Daniel.Beltita@imar.ro}
\date{June 11, 2010}
\makeatother

\begin{abstract} 
We investigate continuity properties of the operators obtained by 
the magnetic Weyl calculus on nilpotent Lie groups, using modulation spaces  associated with 
unitary representations of certain infinite-dimensional Lie groups. 
\end{abstract}

\maketitle


\section{Introduction}

There are three main themes that occur in the present paper: 
\begin{itemize} 
\item[-] The theory of locally convex Lie groups and their representations, recently surveyed in \cite{Ne06}. See also \cite{Ne10}.
\item[-] The pseudo-differential Weyl calculus that takes into account a magnetic field on $\RR^n$, which has been recently developed by techniques of hard analysis, with motivation coming from quantum mechanics; some references in this connection include \cite{MP04}, \cite{IMP07}, and \cite{MP09}. 
\item[-] The modulation spaces from the time-frequency analysis, which have become an increasingly useful tool in the classical pseudo-differential calculus on $\RR^n$; see for instance the seminal paper \cite{GH99}.
\end{itemize}

It follows by our earlier papers \cite{BB09a} and \cite{BB10a} 
that the first two of the above themes are closely related, in the sense that some of the very basic ideas of infinite-dimensional Lie theory 
prove to be very useful for understanding the aforementioned magnetic Weyl calculus as a Weyl quantization of a certain coadjoint orbit of 
a semi-direct product group $M=\Fc\rtimes\RR^n$. 
Here $\Fc$ is a suitable translation-invariant space of smooth functions on $\RR^n$ and the coadjoint orbit is associated 
with a natural unitary representation of $M$ on $L^2(\RR^n)$. 
This representation theoretic approach to the magnetic Weyl calculus is further developed in the present paper by using the third of the themes mentioned above. 
Specifically, we introduce appropriate versions of modulation spaces and use them for describing the continuity properties of the magnetic pseudo-differential operators. 

We recall from \cite{BB09a} that our approach to the magnetic Weyl calculus actually allows us to extend the constructions of \cite{MP04} from the abelian group $(\RR^n,+)$ to any simply connected nilpotent Lie group, 
and this will also be the setting of some of the main results of the present paper. 
However, the proofs are greatly helped by a more general framework that we develop, in the first sections of the paper, for the so-called localized Weyl calculus for representations of 
locally convex Lie groups that satisfy suitable smoothness conditions. 
In order to develop this abstract setting we provide infinite-dimensional extensions of some ideas 
and constructions related to irreducible representations of finite-dimensional nilpotent Lie groups, which we had developed in \cite{BB09c} (see also \cite{BB09b}). 
These extensions may also be interesting on their own, 
however their importance consists in pointing out that the magnetic Weyl calculus of \cite{MP04} and the Weyl-Pedersen calculus initiated in \cite{Pe94} are merely different shapes of the same phenomenon. 

The structure of the paper can be seen form the following table of contents:
\begin{itemize}
\item[\S 1.] Introduction.
\item[\S 2.] Smooth unitary representations of locally convex Lie groups.
\item[\S 3.]  Localized Weyl calculus and modulation spaces.
\item[\S 4.]  Applications to the magnetic Weyl calculus.
\end{itemize}

The aim of sections~2 and~3 is to give general conditions on representations of locally convex Lie groups that ensure good properties of a Weyl calculus and related objects, as Wigner distributions and modulation spaces. 
In fact, in this way we set up a rather general procedure for proving continuity of the operators obtained by the Weyl calculus, and of the Weyl calculus itself. A special case of this procedure, that motivated the present paper, appeared in our earlier work \cite{BB09c} on Weyl-Pedersen calculus for irreducible representations of finite-dimensional nilpotent Lie groups. The developments in this paper allow us to treat the magnetic Weyl calculus as a particular case. In Section~4 we show that the conditions in sections~2 and~3  are met in this case, and 
continuity/trace-class results are thus derived.

\subsection*{Notation}
Throughout the paper we denote by $\Sc(\Vc)$ the Schwartz space 
on a finite-dimensional real vector space~$\Vc$. 
That is, $\Sc(\Vc)$ is the set of all smooth functions 
that decay faster than any polynomial together with 
their partial derivatives of arbitrary order. 
Its topological dual ---the space of tempered distributions on $\Vc$--- 
is denoted by $\Sc'(\Vc)$. 
We use the notation $\Cpol(\Vc)$ for the space 
of smooth functions that grow polynomially together with 
their partial derivatives of arbitrary order; 
the natural locally convex topology of this function space along with some of its special properties are discussed in \cite{Ro75}. 

For every complex vector space $\Yc$ we denote by $\overline{\Yc}$ 
the complex vector space defined by the conditions that $\Yc$ and $\overline{\Yc}$ have the same underlying real vector space, 
and the identity mapping $\Yc\to\overline{\Yc}$ is antilinear. 
If $\Yc$ is a topological vector space, then $\Yc'$ will always denote the weak topological dual of $\Yc$, that is, the space of continuous linear functionals on $\Yc$ endowed with the topology of uniform convergence on the compact subsets.  

We shall always denote by $\cdot\hotimes\cdot$ the completed projective tensor product of locally convex spaces and by $\cdot\botimes\cdot$ the natural tensor product of Hilbert spaces. 
Our references for topological tensor products are \cite{Do74}, \cite{Sch66}, and \cite{Tr67}. 

We shall also use the convention that the Lie groups are denoted by 
upper case Latin letters and the Lie algebras are denoted 
by the corresponding lower case Gothic letters. 

\section{Smooth unitary representations of locally convex Lie groups}\label{Sect2}

Let $M$ be a locally convex Lie group with 
a smooth exponential mapping 
$$\exp_M\colon\Lie(M)=\mg\to M$$ 
(see \cite{Ne06}). 
Assume that $\pi\colon M\to\Bc(\Hc)$ is a unitary representation. 
We denote by $\Hc_\infty$ the space of \emph{smooth vectors} for 
the representation~$\pi$, that is, 
$$\Hc_\infty:=\{\phi\in\Hc\mid\pi(\cdot)\phi\in\Ci(M,\Hc)\}. $$
We note that $\pi(M)\Hc_\infty=\Hc_\infty$ and, 
as proved in \cite[Sect.~IV]{Ne01}, 
the derived representation $\de\pi\colon\mg\to\End(\Hc_\infty)$ 
is well defined and is 
given by 
$$(\forall X\in\mg)(\forall\phi\in\Hc_\infty)\quad
\de\pi(X)\phi=\frac{\de}{\de t}\Big\vert_{t=0}\pi(\exp_M(tX))\phi. $$

\begin{remark}\label{smooth_top}
\normalfont
If we denote by $\U(\mg_{\CC})$ the universal enveloping algebra 
of the complexified Lie algebra $\mg_{\CC}$, 
then the homomorphism of Lie algebras $\de\pi$ extends to a unique homomorphism of unital associative algebras $\de\pi\colon\U(\mg_{\CC})\to\End(\Hc_\infty)$. 
The space of smooth vectors $\Hc_\infty$ will always be considered endowed with the locally convex topology defined by the family of seminorms $\{p_u\}_{u\in\U(\mg_{\CC})}$, 
where for every $u\in\U(\mg_{\CC})$ we define 
$$p_u\colon\Hc_\infty\to[0,\infty),\quad 
p_u(\phi)=\Vert\de\pi(u)\phi\Vert. $$
The inclusion mapping $\Hc_\infty\hookrightarrow\Hc$ 
is continuous and, for all $u\in\U(\mg_{\CC})$ and $m\in M$, 
the linear operators $\de\pi(u)\colon\Hc_\infty\to\Hc_\infty$ and $\pi(m)\colon\Hc_\infty\to\Hc_\infty$ 
are continuous as well. 
\qed
\end{remark}

\begin{definition}\label{smoothness}
\normalfont 
Assume the above setting. 

If the linear subspace of smooth vectors  $\Hc_\infty$ is dense in $\Hc$, 
then the unitary 
representation $\pi\colon M\to\Bc(\Hc)$ is said to be \emph{smooth}.  
If this is the case, then $\pi$ is necessarily continuous, 
in the sense that the group action 
$M\times\Hc\to\Hc$, $(m,f)\mapsto\pi(m)f$, is continuous. 

The representation $\pi$ is said to be \emph{nuclearly smooth} 
if the following conditions are satisfied: 
\begin{enumerate} 
\item\label{smoothness_item1} 
$\pi$ is a smooth representation; 
\item\label{smoothness_item2} 
$\Hc_\infty$ is a nuclear Fr\'echet space; 
\item\label{smoothness_item3} 
both mappings 
$M\times\Hc_\infty\to\Hc_\infty$, $(m,\phi)\mapsto\pi(m)\phi$, 
and $\mg\times\Hc_\infty\to\Hc_\infty$, $(X,\phi)\mapsto\de\pi(X)\phi$ 
are continuous. 
\end{enumerate}
Let $\Bc(\Hc)_\infty$ be the space of 
smooth vectors for the unitary representation 
$$\pi\otimes\bar\pi\colon M\times M\to\Bc(\Sg_2(\Hc)),\quad 
(\pi\otimes\bar\pi)(m_1,m_2)T=\pi(m_1)T\pi(m_2)^{-1}.$$ 

We shall say that the representation $\pi\colon M\to\Bc(\Hc)$ 
is \emph{twice nuclearly smooth} if 
it satisfies the following conditions: 
\begin{enumerate}
\item The representation $\pi$ is nuclearly smooth.  
\item There exists the commutative diagram 
\begin{equation}\label{smoothness_eq1}
\xymatrix{\Hc_\infty\hotimes\overline{\Hc_\infty}
\ar @{^{(}->}[r] \ar[d]
& \Hc\botimes\overline{\Hc} \ar[d] \\
\Bc(\Hc)_\infty \ar @{^{(}->}[r] & \Sg_2(\Hc)
}
\end{equation}
where the vertical arrow on the left is a linear topological isomorphism, 
while the vertical arrow on the right is the natural unitary operator 
defined by the condition 
$(\phi_1,\phi_2)\mapsto\phi_1\otimes\bar\phi_2:=(\cdot\mid\phi_2)\phi_1$. 
\end{enumerate}
\qed
\end{definition}

\begin{remark}
\normalfont
Note that there can exist at most one Fr\'echet topology on $\Hc_\infty$ 
such that the inclusion $\Hc_\infty\hookrightarrow\Hc$ be continuous, 
as a direct consequence of the closed graph theorem. 
\qed
\end{remark}
 
\begin{remark}
\normalfont 
Let $\pi$ be a smooth representation and denote by $\Hc_{-\infty}$ 
the strong dual of $\overline{\Hc_\infty}$. 
Equivalently, $\Hc_{-\infty}$ can be described as the space of continuous antilinear functionals on $\Hc_\infty$ endowed with the topology of uniform convergence on the bounded subsets of $\Hc_\infty$.  
Then there exist the dense embeddings $$\Hc_\infty\hookrightarrow\Hc\hookrightarrow\Hc_{-\infty},$$  
and the duality pairing $(\cdot\mid\cdot)\colon\Hc_{-\infty}\times\Hc_\infty\to\CC$ 
extends the scalar product of~$\Hc$. 
\qed
\end{remark}

\begin{proposition}\label{twice}
If the unitary representation $\pi\colon M\to\Bc(\Hc)$ is twice nuclearly smooth, 
then it also has the following properties: 
\begin{enumerate}
\item\label{twice_item1} 
The representation $\pi\otimes\bar\pi\colon M\times M\to\Bc(\Sg_2(\Hc))$ 
is nuclearly smooth. 
\item\label{twice_item2} 
We have 
$\Lc(\Hc_{-\infty},\Hc_\infty)\simeq\Bc(\Hc)_\infty\hookrightarrow\Sg_1(\Hc)$ 
and there exists the commutative diagram 
$$\xymatrix{\Bc(\Hc) 
\ar @{^{(}->}[r] \ar[d]
& \Lc(\Hc_\infty,\Hc_{-\infty}) \ar[d] \\
\Sg_1(\Hc)' \ar @{^{(}->}[r] & \Lc(\Hc_{-\infty},\Hc_\infty)'
}
$$
where the vertical arrow on the left is the natural linear topological isomorphism defined by the trace duality, 
and the vertical arrow on the right is also a linear topological isomorphism. 
\end{enumerate}
\end{proposition}

\begin{proof}
\eqref{twice_item1} 
The representation $\pi$ is twice nuclearly smooth, 
hence 
$\Hc_\infty$ is a nuclear Fr\'echet space and $\Hc_\infty\hotimes\overline{\Hc_\infty}\simeq\Bc(\Hc)_\infty$. 
Then $\Bc(\Hc)_\infty$ is in turn a nuclear Fr\'echet space 
(see for instance \cite[Prop.~50.1 and Prop.~50.6]{Tr67}).  
Moreover, since $\Hc_\infty$ is dense in $\Hc$, 
it follows that $\Bc(\Hc)_\infty$ is dense in $\Sg_2(\Hc)$. 
To complete the proof of the fact that $\pi\otimes\bar\pi$ is twice nuclearly smooth, we still have to check that the mappings 
$$M\times M\times\Bc(\Hc)_\infty\to\Bc(\Hc)_\infty, \quad 
(m_1,m_2,T)\mapsto\pi(m_1)T\pi(m_2)^{-1}$$
and 
$$\mg\times \mg\times\Bc(\Hc)_\infty\to\Bc(\Hc)_\infty, \quad 
(X_1,X_2,T)\mapsto\de\pi(X_1)T-T\de\pi(X_2)$$
are continuous. 
To this end use again the fact that $\Hc_\infty\hotimes\overline{\Hc_\infty}\simeq\Bc(\Hc)_\infty$ 
and both mappings 
$M\times\Hc_\infty\to\Hc_\infty$, $(m,\phi)\mapsto\pi(m)\phi$, 
and $\mg\times\Hc_\infty\to\Hc_\infty$, $(X,\phi)\mapsto\de\pi(X)\phi$ 
are continuous. 

\eqref{twice_item2} 
Since $\Hc_\infty$ is a nuclear Fr\'echet space, we get 
$$\Lc(\Hc_{-\infty},\Hc_\infty)=\Lc(\overline{\Hc_\infty}',\Hc_\infty)
\simeq \Hc_\infty\hotimes\overline{\Hc_\infty}\simeq\Bc(\Hc)_\infty$$ 
(see \cite[Eq.~(50.17)]{Tr67}). 

Moreover, for every $T\in\Bc(\Hc)_\infty$ we have 
$T\Hc\subseteq\Hc_\infty$. 
Therefore one can prove (as in \cite[Th.~3.3]{BB10b}, for instance) 
that $\Bc(\Hc)_\infty\subseteq\Sg_1(\Hc)$. 
Moreover, by considering the duals of the above topological linear isomorphisms, 
we get 
$$\Lc(\Hc_{-\infty},\Hc_\infty)'\simeq(\Hc_\infty\hotimes\overline{\Hc_\infty})' 
\simeq\Lc(\Hc_\infty,\overline{\Hc_\infty}')
\simeq\Lc(\Hc_\infty,\Hc_{-\infty})$$
(see \cite[Eqs.~(50.19) and (50.16)]{Tr67}), 
and these isomorphisms agree with the isomorphism 
$\Sg_1(\Hc)'\simeq\Bc(\Hc)$ in the sense of the commutative diagram in the statement. 
\end{proof}

\begin{remark}\label{loc6}
\normalfont
For every $f_1,f_2\in\Hc$ we denote by $f_1\otimes\bar f_2\in\Bc(\Hc)$ 
the rank-one operator $f\mapsto(f\mid f_2)f_1$. 
If the representation $\pi\otimes\bar\pi$ is twice nuclearly smooth, 
then for any $f_1,f_2\in\Hc_{-\infty}$ we can use Proposition~\ref{twice} 
to define 
the continuous antilinear functional  
$f_1\otimes\bar f_2\colon\Bc(\Hc)_\infty\to\CC$ 
by $(f_1\otimes\bar f_2)(T)=(f_1\mid Tf_2)$ for every $T\in\Bc(\Hc)_\infty$.
\qed
\end{remark}

\subsection*{Group square}

\begin{definition}\label{loc2}
\normalfont
The \emph{group square} of $M$, 
denoted by $M\ltimes M$, is the semi-direct product defined by 
the action of $M$ on itself by inner automorphisms. 
That is, $M\ltimes M$ is a locally convex Lie group whose underlying manifold is $M\times M$ and the group operation is  
$$(m_1,m_2)(n_1,n_2)=(m_1n_1, n_1^{-1}m_2n_1n_2)$$
for all $m_1,m_2,n_1,n_2\in M$. 
\qed
\end{definition}

\begin{lemma}\label{loc3}
The following assertions hold: 
\begin{enumerate}
\item\label{loc3_item1} 
The mapping 
$$\mu\colon M\ltimes M\to M\times M,\quad 
(m_1,m_2)\mapsto(m_1m_2,m_1)$$
is an isomorphism of Lie groups with tangent map 
$$\Lie(\mu)\colon\mg\ltimes\mg\to\mg\times\mg,\quad  
(X,Y)\mapsto(X+Y,X).$$ 
\item\label{loc3_item2} 
The Lie group $M\ltimes M$ has a smooth exponential map 
$$\exp_{M\ltimes M}\colon\mg\ltimes\mg\to M\ltimes M,\quad 
(X,Y)\mapsto(\exp_MX,\exp_M(-X)\exp_M(X+Y)).$$
\end{enumerate}
\end{lemma}

\begin{proof}
The arguments of Ex.~2.3 in \cite{BB09c} carry over to the present setting. 
\end{proof}

\begin{definition}\label{loc4}
\normalfont
We introduce the continuous unitary representation 
$$\pi^\ltimes\colon M\ltimes M\to\Bc(\Sg_2(\Hc)),\quad 
\pi^\ltimes(m_1,m_2)T=\pi(m_1m_2)T\pi(m_1)^{-1}.$$
To see that $\pi^\ltimes$ is a representation, one can use a direct computation 
or the fact that so is $\pi\otimes\bar\pi$ and we have 
\begin{equation}\label{loc4_eq1}
\pi^\ltimes=(\pi\otimes\bar\pi)\circ\mu,  
\end{equation}
where $\mu\colon M\ltimes M\to M\times M$ is the group isomorphism of Lemma~\ref{loc3}.
\qed
\end{definition}

\section{Localized Weyl calculus and modulation spaces}\label{localized}

The localized  Weyl calculus (see Definition~\ref{loc1} below) 
was introduced in \cite{BB09a} as a tool for dealing with
the magnetic Weyl calculus on nilpotent Lie groups. 
In the present section section we further develop that circle of ideas 
by introducing the modulation spaces 
and extending some related techniques of \cite{BB09c}   to 
the general framework provided by the localized Weyl calculus for 
representations of infinite-dimensional Lie groups. 

Here we single out fairly general conditions that allow for a Weyl calculus to be defined, 
modulation spaces to be considered and continuity properties in these spaces to hold. 
All of these conditions are satisfied in at least two important situations: the Weyl-Pedersen calculus for 
irreducible representations 
of finite dimensional nilpotent Lie groups (see \cite{BB09c}) and the magnetic Weyl calculus of \cite{BB09a} 
to be treated in the last section.

\subsection*{Ambiguity functions and Wigner distributions}

\begin{setting}\label{loc0}
\normalfont
Throughout this section we keep the following notation: 
\begin{enumerate}
\item $M$ is a locally convex Lie group (see \cite{Ne06}) with 
a smooth exponential mapping $\exp_M\colon\Lie(M)=\mg\to M$.
\item $\pi\colon M\to\Bc(\Hc)$ 
is a nuclearly smooth unitary representation. 
\item $\Xi$ and $\Xi^*$ are real finite-dimensional vector spaces 
with a duality pairing 
$\langle\cdot,\cdot\rangle\colon\Xi^*\times\Xi\to\RR$ 
and with Lebesgue measures on $\Xi$ and $\Xi^*$  
suitably normalized for the Fourier transform
$$\widehat{\cdot}
\colon L^1(\Xi)\to L^\infty(\Xi^*), \quad 
b(\cdot)\mapsto\widehat{b}(\cdot)
=\int\limits_{\Xi}\ee^{-\ie\langle\cdot,x\rangle}b(x)\,\de x $$
to give a unitary operator $L^2(\Xi)\to L^2(\Xi^*)$. 
The inverse of this transform will be denoted by $a\mapsto\check a$. 
\end{enumerate}
\qed
\end{setting}

\begin{definition}\label{loc_orthog}
\normalfont
Let $\theta\colon \Xi\to\mg$ be a linear mapping.  

(a) {\it Orthogonality relations.} 
If either $\phi\in\Hc_\infty$ and $f\in\Hc_{-\infty}$, 
or $\phi,f\in\Hc$, then we define the \emph{ambiguity function} along the mapping $\theta$, 
$$\Ac^{\pi,\theta}_\phi f\colon\Xi\to\CC,\quad 
(\Ac^{\pi,\theta}_\phi f)(\cdot)=(f\mid\pi(\exp_M(\theta(\cdot)))\phi). $$
Note that this is a continuous function on $\Xi$.
We say that the representation~$\pi$ satisfies 
the \emph{orthogonality relations} along the mapping $\theta$ if 
\begin{equation}\label{loc_orthog_eq1}
(\Ac^{\pi,\theta}_{\phi_1}f_1\mid\Ac^{\pi,\theta}_{\phi_2}f_2)_{L^2(\Xi)} 
=(f_1\mid f_2)_{\Hc}\cdot(\phi_2\mid\phi_1)_{\Hc}  
\end{equation}
for arbitrary $\phi_1,\phi_2,f_1,f_2\in\Hc$.
In particular, $\Ac^{\pi,\theta}_\phi f\in L^2(\Xi)$ for 
all $\phi,f\in\Hc$. 

(b) {\it Modulation spaces.} Consider any direct sum decomposition $\Xi=\Xi_1\dotplus\Xi_2$ and  
$r,s\in[1,\infty]$. 
For arbitrary $f\in\Hc_{-\infty}$ define 
$$\Vert f\Vert_{M^{r,s}_\phi(\pi,\theta)}
=\Bigl(\int\limits_{\Xi_2}
\Bigl(\int\limits_{\Xi_1}
\vert(\Ac^{\pi,\theta}_\phi f)(X_1,X_2)\vert^r\de X_1 \Bigr)^{s/r}
\de X_2\Bigr)^{1/s}\in[0,\infty] $$
with the usual conventions if $r$ or $s$ is infinite. 
The space 
$$M^{r,s}_\phi(\pi,\theta):=\{f\in\Hc_{-\infty}\mid\Vert f\Vert_{M^{r,s}_\phi(\pi,\theta)}<\infty\}$$ 
is called a \emph{modulation space} for 
the unitary representation $\pi\colon M\to\Bc(\Hc)$ 
with respect to the linear mapping $\theta\colon \Xi\to\mg$, 
the decomposition $\Xi\simeq\Xi_1\times\Xi_2$,  
and the \emph{window vector} $\phi\in\Hc_\infty\setminus\{0\}$. 
\qed
\end{definition}

In connection with the above definition, we note that more general ``co-orbit spaces'' $\Xc_\phi(\pi,\theta)$ can be defined in $\Hc_{-\infty}$ by using any Banach space $\Xc$ of functions on $\Xi$ instead of the mixed-norm Lebesgue spaces $L^{r,s}(\Xi_1\times\Xi_2)$. 
More specifically, one can define for any window vector $\phi\in\Hc_\infty$,
$$\Xc_\phi(\pi,\theta)=\{f\in\Hc_{-\infty}\mid \Ac^{\pi,\theta}_\phi f\in\Xc\}.$$
A systematic investigation of these spaces can be done in a broader context (see \cite{BB10c}).  
However, the modulation spaces $M^{r,s}_\phi(\pi,\theta)$ introduced in Definition~\ref{loc_orthog} above will suffice for the purposes of the present paper. 
See \cite{FG88}, \cite{FG89a}, and \cite{FG89b} for these constructions in the case of representations of locally compact groups.

\begin{remark}\label{mod_L2}
\normalfont 
If the representation~$\pi$ satisfies the orthogonality relations along 
the linear mapping $\theta\colon \Xi\to\mg$, 
then for any decomposition $\Xi=\Xi_1\dotplus\Xi_2$ and 
 any choice of the window vector $\phi\in\Hc_\infty\setminus\{0\}$, 
 we have $M^{2,2}_\phi(\pi,\theta)=\Hc$. 
\qed
\end{remark}

\begin{remark}\label{mod_equiv}
\normalfont 
Let $V\colon\Hc\to\Hc_1$ be a unitary operator and 
consider the unitary representation $\pi_1\colon M\to\Bc(\Hc_1)$ 
such that $V\pi(m)=\pi_1(m)V$ for every $m\in M$. 
Denote by $\Hc_{1,\infty}$ the space of smooth vectors for $\pi_1$ 
and let $\Hc_{1,-\infty}$ be the strong dual of $\overline{\Hc_{1,\infty}}$. 
Then there exist the linear topological isomorphisms 
$V\vert_{\Hc_\infty}\colon\Hc_\infty\to\Hc_{1,\infty}$ and 
$V_{-\infty}\colon\Hc_{-\infty}\to\Hc_{1,-\infty}$, 
where $V_{-\infty}f=f\circ V^*\vert_{\Hc_{1,\infty}}$ 
for every $f\in\Hc_{-\infty}$. 
It is easy to check that for every linear mapping 
$\theta\colon\Xi\to\mg$ and arbitrary $\phi\in\Hc_\infty$ 
and $f\in\Hc_{-\infty}$ we have 
$\Ac^{\pi,\theta}_\phi f=\Ac^{\pi_1,\theta}_{V\phi}(V_{-\infty}f)$. 
Therefore $V_{-\infty}$ naturally gives rise to isometric isomorphisms from the modulation spaces of the representation $\pi$ onto the corresponding modulation spaces of the representation $\pi_1$.
\qed
\end{remark}

\begin{definition}\label{wigner_def}\normalfont
{\it Growth condition.}
We say that the representation~$\pi$  
satisfies the \emph{growth condition} along   
the linear mapping $\theta\colon \Xi\to\mg$
if  
\begin{equation} 
\label{wigner_def_eq0} 
\Ac^{\pi,\theta}_{\phi_2}\phi_1  \in\Sc(\Xi), \quad \text{for all} \; \phi_1,\phi_2\in\Hc_\infty 
 \end{equation}
Note that \eqref{wigner_def_eq0}  implies that the sesquilinear map 
$$\Ac^{\pi,\theta}\colon\Hc_\infty\times\Hc_\infty\to\Sc(\Xi),\quad 
(\phi_1,\phi_2)\mapsto\Ac^{\pi,\theta}_{\phi_2}\phi_1$$
is separately continuous as a straightforward application of the closed graph theorem,  
and then it is jointly continuous by \cite[Cor.~1 to Th.~5.1 in Ch.~III]{Sch66}. 

If the representation~$\pi$ satisfies the orthogonality relations along the mapping~$\theta$, 
and $\phi,f\in\Hc$, then $\Ac^{\pi,\theta}_\phi f\in L^2(\Xi)$, 
hence we can define the \emph{cross-Wigner distribution} 
$\Wig(f,\phi)\in L^2(\Xi^*)$ by the condition $\widehat{\Wig(f,\phi)}:=\Ac^{\pi,\theta}_\phi f$.
\qed
\end{definition}

\begin{definition}\label{density_cond}
\normalfont
{\it Density condition.}
The representation~$\pi$ is said to 
satisfy the \emph{density condition} along   
the linear mapping $\theta\colon \Xi\to\mg$
if $\{\Ac^{\pi,\theta}_\phi f\mid\phi,f\in\Hc\}$ is a total subset of $L^2(\Xi)$, 
in the sense that it spans a dense linear subspace. 
\qed
\end{definition}

\begin{remark}
\normalfont
If the representation $\pi$ satisfies the orthogonality relations along $\theta$, then it follows in particular that 
$\{\Ac^{\pi,\theta}_\phi f\mid\phi,f\in\Hc\}\subseteq L^2(\Xi)$, 
however it is not clear in general that this subset of $L^2(\Xi)$ is total. 
Similarly, if $\pi$ satisfies the growth condition along $\theta$, 
then $\{\Ac^{\pi,\theta}_\phi f
\mid\phi,f\in\Hc_\infty\}\subseteq\Sc(\Xi)\subseteq L^2(\Xi)$, however in this way we may not get a total subset of $L^2(\Xi)$. 
\qed
\end{remark}

\begin{lemma}\label{loc5}
If the representation~$\pi$ satisfies the orthogonality relations along the linear mapping $\theta\colon\Xi\to\mg$, 
then the following assertions hold: 
\begin{enumerate}
\item\label{loc5_item1} 
The representation~$\pi\otimes\bar\pi$  
satisfies the orthogonality relations along  the linear mapping $\theta\times\theta\colon\Xi\times\Xi\to\mg\times\mg$.
\item\label{loc5_item2} 
The representation~$\pi^\ltimes$ satisfies the orthogonality relations
along each of the linear mappings $\Lie(\mu)^{-1}\circ(\theta\times\theta)\colon\Xi\times\Xi\to\mg\ltimes\mg$ 
and $\theta\times\theta\colon\Xi\times\Xi\to\mg\ltimes\mg$.
\end{enumerate}
\end{lemma}

\begin{proof}
To see that Assertion~\eqref{loc5_item1} holds, 
first prove the orthogonality relations for rank-one operators in $\Sg_2(\Hc)$, 
then extend them by sesquilinearity to the finite-rank operators, 
and eventually extend them by continuity to arbitrary Hilbert-Schmidt operators. 
Then Assertion~\eqref{loc5_item2} on $\Lie(\mu)\circ(\theta\times\theta)$ follows by Assertion~\eqref{loc5_item1} 
along with equation~\eqref{loc4_eq1}. 

Then, to see that also the representation~$\pi^\ltimes$ satisfies the orthogonality relations
along
$\theta\times\theta\colon\Xi\times\Xi\to\mg\ltimes\mg$, 
just note that 
$$(\Lie(\mu)^{-1}\circ(\theta\times\theta))(X,Y)=(\theta(Y),\theta(X)-\theta(Y))
=(\theta\times\theta)(Y,X-Y) $$ 
and the linear mapping 
$\Xi\times\Xi\to\Xi\times\Xi$, $(X,Y)\mapsto(Y,X-Y)$,  
has the Jacobian identically equal to~1. 
\end{proof}

\begin{lemma}\label{loc55}
If the representation~$\pi$ satisfies the growth condition along the linear mapping $\theta\colon\Xi\to\mg$, 
then the following assertions hold: 
\begin{enumerate}
\item\label{loc55_item1} 
The representation~$\pi\otimes\bar\pi$  
satisfies the growth condition along  the linear mapping $\theta\times\theta\colon\Xi\times\Xi\to\mg\times\mg$.
\item\label{loc55_item2} 
The representation~$\pi^\ltimes$ satisfies the growth condition
along each of the linear mappings $\Lie(\mu)^{-1}\circ(\theta\times\theta)\colon\Xi\times\Xi\to\mg\ltimes\mg$ 
and $\theta\times\theta\colon\Xi\times\Xi\to\mg\ltimes\mg$.
\end{enumerate}
\end{lemma}

\begin{proof}
The growth condition for the representation $\pi$ along $\theta$ implies that the bilinear map
$ \Ac^{\pi, \theta}\colon \Hc_\infty\times \overline{\Hc_\infty}\to \Sc (\Xi)$
is continuous, hence extends to a continuous linear map
$$ \Ac^{\pi, \theta}\colon \Hc_\infty\hotimes \overline{\Hc_\infty}\to \Sc (\Xi).$$
By complex conjugation we also have
$$ \overline{\Ac^{\pi, \theta}}\colon \overline{\Hc_\infty\hotimes \overline{\Hc_\infty}}=
\overline{\Hc_\infty}\hotimes \Hc_\infty \to \Sc (\Xi).$$
Thus we get the continuous mapping
$$ \Ac^{\pi, \theta} \hotimes \overline{\Ac^{\pi, \theta}}\colon 
\Hc_\infty\hotimes \overline{\Hc_\infty}
\hotimes \overline{\Hc_\infty}\hotimes \Hc_\infty \to \Sc (\Xi)\hotimes \Sc (\Xi)=\Sc (\Xi\times \Xi).
$$
By composing this with the permutation 
$(f_1, \phi_1, f_2, \phi_2) \mapsto (f_1, f_2, \phi_1, \phi_2)$ 
and using that $\Hc_\infty\hotimes \overline{\Hc_\infty}\simeq \Bc(\Hc)_\infty$, 
we get a continuous operator $\Bc(\Hc)_\infty\hotimes \overline{\Bc(\Hc)_\infty} \to \Sc (\Xi\times \Xi)$
which extends $\Ac^{\pi\otimes \bar{\pi}, \theta\times \theta}$, since
$$\Ac^{\pi\otimes\bar\pi,\theta\times\theta}_{\phi_1\otimes\bar\phi_2}(f_1\otimes\bar f_2)
=\Ac^{\pi,\theta}_{\phi_1}f_1\otimes
\overline{\Ac^{\pi,\theta}_{\phi_2}f_2}.$$
The second part in the growth condition can be checked similarly, by using that $\Hc_{-\infty}$ is nuclear, like
$\Hc_{\infty}$ (see \cite[Ch.IV, Th. 9.6]{Sch66}), and noting the isomorphisms $\Hc_{-\infty}\hotimes \overline{\Hc_{-\infty}}\simeq (\overline{\Hc_{\infty}}\hotimes \Hc_{\infty})'\simeq\overline{\Bc(\Hc)_\infty}'$.

Assertion~\eqref{loc55_item2} on $\Lie(\mu)\circ(\theta\times\theta)$ follows by Assertion~\eqref{loc55_item1} 
along with equation~\eqref{loc4_eq1}. 

Then, to see that also the representation~$\pi^\ltimes$ satisfies the growth condition
along
$\theta\times\theta\colon\Xi\times\Xi\to\mg\ltimes\mg$, 
just note that 
$$(\Lie(\mu)^{-1}\circ(\theta\times\theta))(X,Y)=(\theta(Y),\theta(X)-\theta(Y))
=(\theta\times\theta)(Y,X-Y) $$ 
and the linear mapping 
$\Xi\times\Xi\to\Xi\times\Xi$, $(X,Y)\mapsto(Y,X-Y)$,  
is invertible. 
\end{proof}

\subsection*{Localized Weyl calculus and its continuity properties}

\begin{definition}\label{loc1}
\normalfont 
Let $\theta\colon \Xi\to\mg$ be a linear mapping.

The \emph{localized Weyl calculus for $\pi$ along~$\theta$} 
is the mapping $\Op^\theta\colon\widehat{L^1(\Xi)}\to\Bc(\Hc)$ given by 
\begin{equation}\label{loc1_eq1}
\Op^\theta(a)
=\int\limits_{\Xi} \check{a}(X)\pi(\exp_M(\theta(X)))\,\de X
\end{equation}
for $a\in\widehat{L^1(\Xi)}$ where we use weakly convergent integrals.

The localized Weyl calculus for $\pi$ along $\theta$ 
is said to be \emph{regular} if 
\begin{itemize}
\item $\pi$ satisfies the growth condition along the mapping $\theta$, 
\item $\pi$ is twice nuclearly smooth, and 
\item $\Op^\theta(a) \in \Bc(\Hc)_\infty$ whenever  $a\in \Sc(\Xi^*)$. 
\end{itemize}
Note that the closed graph theorem then implies that $\Op^\theta\colon\Sc(\Xi^*)\to\Bc(\Hc)_\infty$ 
is a continuous linear mapping. 
\qed
\end{definition}

If the representation~$\pi$ satisfies the growth condition along the mapping $\theta$, then one can think of~\eqref{loc1_eq1} in the distributional sense in order to define the localized Weyl calculus 
$\Op^\theta\colon\Sc'(\Xi^*)\to\Lc(\Hc_\infty,\Hc_{-\infty})$. 
More specifically, for every $a\in\Sc'(\Xi^*)$ 
and $\phi,\psi\in\Hc_\infty$ we have 
\begin{equation}\label{loc1_eq2}
(\Op^\theta(a)\phi\mid\psi)
=\langle\check{a},\overline{\Ac^{\pi,\theta}_\phi\psi}\rangle 
\end{equation}
where $\langle\cdot,\cdot\rangle\colon\Sc'(\Xi)\times\Sc(\Xi)\to\CC$ 
is the usual duality pairing.

\begin{remark}\label{extension}
\normalfont
If the localized Weyl calculus for $\pi$ along $\theta$ 
is regular and moreover defines a linear topological isomorphism
$\Op^\theta\colon\Sc(\Xi^*)\to\Bc(\Hc)_\infty$ 
(see Proposition~\ref{unitary} for sufficient conditions), then  
we also have the linear topological isomorphism 
$\Op^\theta\colon\Sc'(\Xi^*)\to\Lc(\Hc_\infty,\Hc_{-\infty})$ 
by~Proposition~\ref{twice}\eqref{twice_item2}. 
Therefore, by using Remark~\ref{loc6}, 
we see that there exist the sesquilinear mappings 
\begin{equation}\label{extension_eq1}
\Ac^{\pi,\theta}\colon\Hc_{-\infty}\times\Hc_{-\infty}\to\Sc'(\Xi) 
\text{ and }
\Wig\colon\Hc_{-\infty}\times\Hc_{-\infty}\to\Sc'(\Xi^*)
\end{equation}
such that 
$$\Op^\theta(\Wig(f_1,f_2))=f_1\otimes\bar f_2$$ 
and $\widehat{\Wig(f_1,f_2)}=\Ac^{\pi,\theta}_{f_2}f_1$
for all $f_1,f_2\in\Hc_{-\infty}$. 
In addition, it follows by \eqref{loc1_eq2} and the definition of the Fourier transform for tempered distributions that 
for every $a\in\Sc'(\Xi^*)$ and $\phi,\psi\in\Hc_\infty$ we have 
\begin{equation}\label{loc1_eq2_bis}
(\Op^\theta(a)\phi\mid\psi)=(a\mid\Wig(\psi,\phi)). 
\end{equation}
If moreover the representation~$\pi$ satisfies the orthogonality relations along the linear mapping $\theta$, 
then it follows by Proposiion~\ref{unitary} below that the mappings~\eqref{extension_eq1}
agree with the ambiguity functions and the cross-Wigner distributions 
(see Definition~\ref{wigner_def}).
\qed
\end{remark}

\begin{proposition}\label{unitary}
If $\pi$ satisfies the orthogonality relations along the linear mapping $\theta\colon\Xi\to\mg$, then the following assertions are equivalent: 
\begin{enumerate}
\item\label{unitary_item1}
The representation $\pi$ satisfies the density condition along $\theta$. 
\item\label{unitary_item2} 
There exists a unique unitary operator $\Op^\theta\colon L^2(\Xi^*)\to\Sg_2(\Hc)$ which agrees with the localized Weyl calculus for $\pi$ along~$\theta$. 
\end{enumerate}
If these assertions hold true, then we have 
\begin{equation}\label{unitary_eq1}
(\forall f,\phi\in\Hc)\quad \Op^\theta(\Wig(f,\phi))=f\otimes\bar\phi.
\end{equation}
If moreover the localized Weyl calculus for $\pi$ along~$\theta$ is regular, then the mapping $\Op^\theta\colon\Sc(\Xi^*)\to\Bc(\Hc)_\infty$ 
is a linear topological isomorphism. 
\end{proposition}

\begin{proof}
We begin with some general remarks. 
Since we have a unitary Fourier transform $L^2(\Xi)\to L^2(\Xi^*)$, 
it follows by the orthogonality relations along with~\eqref{loc1_eq2} that 
for arbitrary $f,\phi\in\Hc$ we have 
\begin{equation}\label{unitary_proof_eq1}
\Op^\theta(\Wig(f,\phi))=f\otimes\bar\phi\text{ and }
\Vert\Wig(f,\phi)\Vert_{L^2(\Xi^*)}
=\Vert f\Vert\cdot\Vert\phi\Vert
=\Vert f\otimes\bar\phi\Vert_{\Sg_2(\Hc)}.
\end{equation}
Moreover, 
\begin{equation}\label{unitary_proof_eq2}
\spa(\{f\otimes\bar\phi\mid f,\phi\in\Hc\})
\text{ is dense in }\Sg_2(\Hc).
\end{equation}
We now come back to the proof. 

``\eqref{unitary_item1}$\Rightarrow$\eqref{unitary_item2}'' 
Let $\pi$ satisfy the density condition along $\theta$. 
Since the Fourier transform $L^2(\Xi)\to L^2(\Xi^*)$ is unitary, 
it follows that $\spa(\{\Wig(f,\phi)\mid f,\phi\in\Hc\})$ is a dense linear subspace of~$L^2(\Xi^*)$. 
Therefore, by using \eqref{unitary_proof_eq1} and \eqref{unitary_proof_eq2}, 
we see that $\Op^\theta$ uniquely extends to a unitary operator 
$L^2(\Xi^*)\to\Sg_2(\Hc)$. 

``\eqref{unitary_item2}$\Rightarrow$\eqref{unitary_item1}'' 
If the operator $\Op^\theta\colon L^2(\Xi^*)\to\Sg_2(\Hc)$ is unitary, 
then it follows by \eqref{unitary_proof_eq1} and \eqref{unitary_proof_eq2} 
that $\spa(\{\Wig(f,\phi)\mid f,\phi\in\Hc\})$ is a dense linear subspace of~$L^2(\Xi^*)$. 
Then, by using again the fact that the Fourier transform $L^2(\Xi)\to L^2(\Xi^*)$ is unitary, we can see that 
$\spa(\{\Ac^{\pi,\theta}_\phi f\mid f,\phi\in\Hc\})$ is a dense linear subspace of~$L^2(\Xi)$, 
that is, $\pi$ satisfies the density condition along $\theta$. 

Now assume that the assertions \eqref{unitary_item1} and \eqref{unitary_item2} in the statement 
are satisfied and the localized Weyl calculus for $\pi$ along $\theta$ is regular. 
Then $\pi$ satisfies the growth condition along $\theta$, 
hence the ambiguity function defines a continuous sesquilinear mapping $\Ac^{\pi,\theta}\colon\Hc_\infty\times\Hc_\infty\to\Sc(\Xi)$ 
(see Definition~\ref{wigner_def}). 
Since the Fourier transform is a linear topological isomorphism 
$\Sc(\Xi)\to\Sc(\Xi^*)$, the cross-Wigner distributions 
also define a continuous sesquilinear mapping $\Wig\colon\Hc_\infty\times\Hc_\infty\to\Sc(\Xi^*)$, 
which further induces a continuous linear mapping 
$\Wig\colon\Hc_\infty\hotimes\overline{\Hc_\infty}\to\Sc(\Xi^*)$. 
On the other hand, the condition that the localized Weyl calculus for $\pi$ along $\theta$ is regular (see Definition~\ref{loc1}) 
includes the assumption that the representation $\pi$ is twice nuclearly smooth, 
hence we have a topological linear isomorphism $\Hc_\infty\hotimes\overline{\Hc_\infty}\simeq\Bc(\Hc)_\infty$. 

We thus eventually get a continuous linear mapping 
$\Wig\colon\Bc(\Hc)_\infty\to\Sc(\Xi^*)$ which, by \eqref{unitary_proof_eq1}, 
has the property $\Op^\theta\circ\Wig=\id$ on $\Bc(\Hc)_\infty$. 
In other words, $\Wig=(\Op^\theta)^{-1}\mid_{\Bc(\Hc)_\infty}$. 
Thus the unitary operator $\Op^\theta\colon L^2(\Xi^*)\to\Sg_2(\Hc)$ 
restricts to a continuous linear map $\Sc(\Xi^*)\to\Bc(\Hc)_\infty$ 
(since the localized Weyl calculus for $\pi$ along $\theta$ is regular), 
while its inverse $(\Op^\theta)^{-1}$ restricts to 
a continuous linear map $\Wig\colon\Bc(\Hc)_\infty\to\Sc(\Xi^*)$. 
It then follows that $\Op^\theta\colon\Sc(\Xi^*)\to\Bc(\Hc)_\infty$ 
is a linear topological isomorphism (whose inverse is $\Wig$). 
\end{proof}

\begin{definition}\label{moyal}
\normalfont
Assume that the localized Weyl calculus for $\pi$ along the linear mapping $\theta\colon\Xi\to\mg$ is regular and the representation $\pi$ satisfies both the density condition and the orthogonality relations along $\theta$.  
It follows by Proposition~\ref{unitary} that the localized Weyl calculus $\Op^\theta$ 
defines a unitary operator 
$L^2(\Xi^*)\to\Sg_2(\Hc)$, 
and also linear topological isomorphisms $\Sc(\Xi^*)\to\Bc(\Hc)_\infty\simeq\Lc(\Hc_{-\infty},\Hc_\infty)$ 
and $\Sc'(\Xi^*)\to\Lc(\Hc_\infty,\Hc_{-\infty})$. 
Hence we can introduce the following notions:
\begin{enumerate}
\item If $a,b\in\Sc'(\Xi^*)$ and there exists the well-defined 
the operator product 
$\Op^\theta(a)\Op^\theta(b)\in\Lc(\Hc_\infty,\Hc_{-\infty})$, then Remark~\ref{extension} shows that 
the \emph{Moyal product} $a\#^\theta b\in\Sc'(\Xi^*)$ 
is uniquely determined by the condition 
$$\Op^\theta(a\#^\theta b)=\Op^\theta(a)\Op^\theta(b).$$ 
Thus the Moyal product defines bilinear mappings 
$\Sc(\Xi^*)\times\Sc(\Xi^*)\to\Sc(\Xi^*)$ and 
$L^2(\Xi^*)\times L^2(\Xi^*)\to L^2(\Xi^*)$.
\item We define the unitary representation 
$\pi^\#\colon M\ltimes M\to\Bc(L^2(\Xi^*))$  
such that for every $m\in M\ltimes M$ there exists the commutative diagram 
$$\xymatrix{
L^2(\Xi^*)\ar[r]^{\pi^\#(m)} \ar[d]_{\Op^\theta} & L^2(\Xi^*) \ar[d]^{\Op^\theta} \\
\Sg_2(\Hc)\ar[r]^{\pi^\ltimes(m)} & \Sg_2(\Hc) 
}$$
\end{enumerate} 
These constructions provide extensions of some notions introduced in \cite{BB09c}. 
\qed
\end{definition}

\begin{remark}\label{moyal_rem}
\normalfont
In the setting of Definition~\ref{moyal} we note the following facts: 
\begin{enumerate}
\item For every $m_1,m_2\in M$ and $f\in L^2(\Xi^*)$ we have 
$$\pi^\#(m_1,m_2)f=(\Op^\theta)^{-1}(\pi(m_1m_2))\#^\theta 
 f\#^\theta(\Op^\theta)^{-1}(\pi(m_1))^{-1}.$$
\item For every  $X_1,X_2\in\Xi$ we have 
$\Op^\theta(\ee^{\ie\langle\cdot,X_j\rangle})=\pi(\exp_M(\theta(X_j)))$ for $j=1,2$, whence 
by Lemma~\ref{loc3}\eqref{loc3_item2}
$$\begin{aligned}
\pi^\#(\exp_{M\ltimes M} &(\theta(X_1),\theta(X_2))f \\
&=\pi^\#(\exp_M(\theta(X_1)),\exp_M(-\theta(X_1))\exp_M(\theta(X_1+X_2)))f \\
&=\ee^{\ie\langle\cdot,X_1+X_2\rangle}\#^\theta f\#^\theta \ee^{-\ie\langle\cdot,X_1\rangle}
\end{aligned} $$
whenever $f\in L^2(\Xi^*)$.
\end{enumerate}
\qed
\end{remark}

\begin{proposition}\label{loc7}
Assume that the representation $\pi$ is twice nuclearly smooth. 
If we have either $\phi_1,\phi_2,f_1,f_2\in\Hc$, 
or $\phi_1,\phi_2\in\Hc_\infty$ and $f_1,f_2\in\Hc_{-\infty}$, 
then 
$$(\forall X,Y\in\Xi)\quad 
(\Ac^{\pi^\ltimes,\theta\times\theta}_{\phi_1\otimes\bar\phi_2}(f_1\otimes\bar f_2))(X,Y)
=(\Ac^{\pi,\theta}_{\phi_1}f_1)(X+Y)\cdot
\overline{(\Ac^{\pi,\theta}_{\phi_2}f_2)(X)}.$$
If moreover the localized Weyl calculus for $\pi$ along  
$\theta$ is regular and the representation $\pi$ satisfies both the density condition and the orthogonality relations along $\theta$, 
then
$$(\forall X,Y\in\Xi)\quad 
(\Ac^{\pi^\#,\theta\times\theta}_{\Wig(\phi_1,\phi_2)}(\Wig(f_1,f_2)))(X,Y)
=(\Ac^{\pi,\theta}_{\phi_1}f_1)(X+Y)\cdot
\overline{(\Ac^{\pi,\theta}_{\phi_2}f_2)(X)}.$$
\end{proposition}

\begin{proof}
It follows at once by definition that 
$$\Ac^{\pi\otimes\bar\pi,\theta\times\theta}_{\phi_1\otimes\bar\phi_2}(f_1\otimes\bar f_2)
=\Ac^{\pi,\theta}_{\phi_1}f_1\otimes
\overline{\Ac^{\pi,\theta}_{\phi_2}f_2}.$$
On the other hand,  we easily get by \eqref{loc4_eq1} 
$$(\forall X,Y\in\Xi)\quad 
(\Ac^{\pi^\ltimes,\theta\times\theta}_{\phi_1\otimes\bar\phi_2}(f_1\otimes\bar f_2))
(X,Y)
= (\Ac^{\pi\otimes\bar\pi,\theta\times\theta}_{\phi_1\otimes\bar\phi_2}(f_1\otimes\bar f_2))
(X+Y,X).$$
For the second part of the statement, just recall that 
$\Op^\theta(\Wig(f_1,f_2))=f_1\otimes\bar f_2$ and use 
Proposition~\ref{unitary} along with Remark~\ref{mod_equiv}. 
\end{proof}

The next theorem extends a result in \cite{BB09c}, and the general lines of the proof go back to \cite{To04}. 
\begin{theorem}\label{wigner_cont_th}
Let $\phi_1,\phi_2\in\Hc_\infty\setminus\{0\}$, 
and assume the following hypotheses: 
\begin{enumerate}
\item 
The representation $\pi$ satisfies both the density condition and the orthogonality relations  
along the linear mapping $\theta\colon\Xi\to\mg$. 
\item The localized Weyl calculus for the representation $\pi$ along $\theta$ is regular.  
\end{enumerate}
Now let $\Xi=\Xi_1\dotplus\Xi_2$ be any direct sum decomposition. 
If $1\le r\le s\le\infty$ and $r_1,r_2,s_1,s_2\in[r,s]$ satisfy 
the equations 
$\frac{1}{r_1}+\frac{1}{r_2}
=\frac{1}{s_1}+\frac{1}{s_2}
=\frac{1}{r}+\frac{1}{s}$, 
then the cross-Wigner distribution 
defines a continuous sesquilinear map 
$$\Wig(\cdot,\cdot)\colon 
M^{r_1,s_1}_{\phi_1}(\pi,\theta)\times M^{r_2,s_2}_{\phi_2}(\pi,\theta) 
\to M^{r,s}_{\Wig(\phi_1,\phi_2)}(\pi^\#,\theta\times\theta).$$
\end{theorem}

\begin{proof} 
First recall from Proposition~\ref{unitary} that the localized Weyl calculus for the representation $\pi$ along $\theta$ defines a unitary operator 
$\Op^\theta\colon L^2(\Xi^*)\to\Sg_2(\Hc)$.

Let $f_1,f_2\in\Hc_{-\infty}$ and note that for every $X\in\Xi$ we have 
\begin{equation}\label{wigner_cont_th_Delta}
\overline{(\Ac^{\pi,\theta}_{\phi_2}f_2)(X)}
=\overline{(f_2\mid\pi(\exp_M( \theta(X)))\phi_2)}
=(\Ac^{\pi,\theta}_{f_2}\phi_2)(-X).
\end{equation}
Therefore by Proposition~\ref{loc7} we get 
\begin{equation}\label{wigner_cont_th_zero}
\Vert\Wig(f_1,f_2)\Vert_{M^{r,s}_{\Wig(\phi_1,\phi_2)}(\pi^\#,\theta\times\theta)}
=\Bigl(\int\limits_{\Xi_2}F(Y_2)\de Y_2\Bigr)^{1/s},
\end{equation}
where 
\begin{equation}\label{wigner_cont_th_star}
\begin{aligned}
F(Y_2)=\int\limits_{\Xi_1}\Bigl(\int\limits_{\Xi_2}\int\limits_{\Xi_1}
\vert & (\Ac^{\pi,\theta}_{\phi_1}f_1)(X_1+Y_1,X_2+Y_2) \\
&\times(\Ac^{\pi,\theta}_{f_2}\phi_2)(-X_1,-X_2)\vert^r
\de X_1\de X_2\Bigr)^{s/r} \de Y_1.
\end{aligned}
\end{equation}
On the other hand, it follows by Minkowski's inequality that for every measurable function 
$\Gamma\colon\Xi_1\times\Xi_2\times\Xi_2\to\CC$ and every real number $t\ge1$ we have 
\begin{equation}\label{wigner_cont_th_sstar}
\Bigl(\int\limits_{\Xi_1}\Bigl(\int\limits_{\Xi_2}\vert\Gamma(Y_1,X_2,Y_2)\vert\de X_2\Bigr)^t 
\de Y_1\Bigr)^{1/t}
\le\int\limits_{\Xi_2}\Bigl(\int\limits_{\Xi_1}\vert\Gamma(Y_1,X_2,Y_2)\vert^t\de Y_1\Bigr)^{1/t}\de X_2
\end{equation}
whenever $Y_2\in\Xi_2$. 
By \eqref{wigner_cont_th_star} and \eqref{wigner_cont_th_sstar} with $t:=s/r$ 
and 
$$\Gamma(Y_1,X_2,Y_2):=\int\limits_{\Xi_1}
\vert(\Ac^{\pi,\theta}_{\phi_1}f_1)(Y_1-X_1,Y_2-X_2)
\cdot(\Ac^{\pi,\theta}_{f_2}\phi_2)(X_1,X_2)\vert^r\de X_1$$ 
we get 
\begin{equation}\label{wigner_cont_th_ssstar}
\begin{aligned}
F(Y_2)
\le &\Bigl(\int\limits_{\Xi_2}\Bigl(\int\limits_{\Xi_1}\Gamma(Y_1,X_2,Y_2)^{s/r}
      \de Y_1\Bigr)^{r/s}
      \de X_2\Bigr)^{s/r} \\
    =&\Bigl(\int\limits_{\Xi_2}\Vert\Gamma(\cdot,X_2,Y_2)\Vert_{L^{s/r}(\Xi_1)}
      \de X_2\Bigr)^{s/r}. 
\end{aligned} 
\end{equation}
Now note that $\Gamma(\cdot,X_2,Y_2)$ is equal to the convolution product of the functions 
$\vert(\Ac^{\pi,\theta}_{\phi_1}f_1)(\cdot,Y_2-X_2)\vert^r$ and $\vert(\Ac^{\pi,\theta}_{f_2}\phi_2)(\cdot,X_2)\vert^r$. 
It follows by Young's inequality that 
$$\begin{aligned}
\Vert\Gamma(\cdot,X_2,Y_2)\Vert_{L^{s/r}(\Xi_1)}
&\le\Vert\vert(\Ac^{\pi,\theta}_{\phi_1}f_1)(\cdot,Y_2-X_2)
 \vert^r\Vert_{L^{t_1}(\Xi_1)}
     \Vert\vert(\Ac^{\pi,\theta}_{f_2}\phi_2)(\cdot,X_2)
     \vert^r\Vert_{L^{t_2}(\Xi_1)} \\
&=\Vert(\Ac^{\pi,\theta}_{\phi_1}f_1)(\cdot,Y_2-X_2)
   \Vert_{L^{rt_1}(\Xi_1)}^r
     \Vert(\Ac^{\pi,\theta}_{f_2}\phi_2)(\cdot,X_2)\Vert_{L^{rt_2}(\Xi_1)}^r
\end{aligned}$$
whenever $t_1,t_2\in[1,\infty]$ satisfy $\frac{1}{t_1}+\frac{1}{t_2}=1+\frac{r}{s}$. 
By using the above inequality with $t_j=\frac{r_j}{r}$ for $j=1,2$, 
and taking into account \eqref{wigner_cont_th_ssstar}, we get 
\begin{equation}\label{wigner_cont_th_sssstar}
\begin{aligned}
F(Y_2)\le
  &\Bigl(\int\limits_{\Xi_2}
        \Vert(\Ac^{\pi,\theta}_{\phi_1}f_1)(\cdot,Y_2-X_2)
          \Vert_{L^{rt_1}(\Xi_1)}^r
     \Vert(\Ac^{\pi,\theta}_{f_2}\phi_2)(\cdot,X_2)\Vert_{L^{rt_2}(\Xi_1)}^r
      \de X_2\Bigr)^{s/r} \\
  &=:\theta(Y_2)^{s/r}, 
\end{aligned}
\end{equation}
where $\theta(\cdot)$ is the convolution of the functions 
$X_2\mapsto \Vert(\Ac^{\pi,\theta}_{\phi_1}f_1)(\cdot,X_2)\Vert_{L^{rt_1}(\Xi_1)}^r$ 
and 
$X_2\mapsto\Vert(\Ac^{\pi,\theta}_{f_2}\phi_2)(\cdot,X_2)\Vert_{L^{rt_2}(\Xi_1)}^r$. 
It follows by Young's inequality again that 
$$\begin{aligned}
\Vert\theta\Vert_{L^{s/r}(\Xi_2)}
\le &
\Bigl(\int\limits_{\Xi_2}\Vert(\Ac^{\pi,\theta}_{\phi_1}f_1)(\cdot,X_2)\Vert_{L^{rt_1}(\Xi_1)}^r\de X_2\Bigr)^{1/m_1} \\
&\times\Bigl(\int\limits_{\Xi_2}\Vert(\Ac^{\pi,\theta}_{f_2}\phi_2)(\cdot,X_2)\Vert_{L^{rt_2}(\Xi_1)}^r\de X_2\Bigr)^{1/m_2} 
\end{aligned}$$
provided that $m_1,m_2\in[1,\infty]$ and $\frac{1}{m_1}+\frac{1}{m_2}=1+\frac{r}{s}$. 
For $m_j=\frac{s_j}{r}$, $j=1,2$, we get 
$$\Vert\theta\Vert_{L^{s/r}(\Xi_2)}
\le (\Vert f_1\Vert_{M^{r_1,s_1}_{\phi_1}(\pi,\theta)})^r (\Vert f_2\Vert_{M^{r_2,s_2}_{\phi_2}(\pi,\theta)})^r, $$
where we also used \eqref{wigner_cont_th_Delta}. 
Then by \eqref{wigner_cont_th_zero} and \eqref{wigner_cont_th_sssstar} 
we get 
$$\Vert\Wig(f_1,f_2)\Vert_{M^{r,s}_{\Wig(\phi_1,\phi_2)}(\pi^\#,\theta\times\theta)}
\le\Vert f_1\Vert_{M^{r_1,s_1}_{\phi_1}(\pi,\theta)}\cdot
\Vert f_2\Vert_{M^{r_2,s_2}_{\phi_2}(\pi,\theta)},$$ 
and this concludes the proof. 
\end{proof}

\begin{corollary}\label{C2}
Let $\phi_1,\phi_2\in\Hc_\infty\setminus\{0\}$, 
and assume the following hypotheses: 
\begin{enumerate}
\item 
The representation $\pi$ satisfies both the density condition and the orthogonality relations  
along the linear mapping $\theta\colon\Xi\to\mg$. 
\item The localized Weyl calculus for the representation $\pi$ along $\theta$ is regular.  
\end{enumerate}
Now let $\Xi=\Xi_1\dotplus\Xi_2$ be any direct sum decomposition.
If $r,s,r_1,s_1,r_2,s_2\in[1,\infty]$ satisfy the conditions  
$$r\le s, \quad r_2,s_2\in[r,s],\quad\text{and}\quad
\frac{1}{r_1}-\frac{1}{r_2}=\frac{1}{s_1}-\frac{1}{s_2}=1-\frac{1}{r}-\frac{1}{s},$$
then for every symbol 
$a\in M^{r,s}_{\Wig(\phi_1,\phi_2)}(\pi^\#,\theta\times\theta)$ we have 
a bounded linear operator 
$$\Op^\theta(a)\colon M^{r_1,s_1}_{\phi_1}(\pi,\theta)\to M^{r_2,s_2}_{\phi_2}(\pi,\theta).$$ 
Moreover, the linear mapping 
$$\Op^\theta\colon  M^{r,s}_{\Wig(\phi_1,\phi_2)}(\pi^\#,\theta\times\theta)
\to\Bc(M^{r_1,s_1}_{\phi_1}(\pi,\theta),M^{r_2,s_2}_{\phi_2}(\pi,\theta))$$ 
is continuous. 
\end{corollary}

\begin{proof} 
For every $t\in[1,\infty]$ define $t'\in[1,\infty]$ by the equation $\frac{1}{t}+\frac{1}{t'}=1$. 
With this notation, the hypothesis implies 
$\frac{1}{r_1}+\frac{1}{r_2'}=\frac{1}{s_1}+\frac{1}{s_2'}=\frac{1}{r'}+\frac{1}{s'}$ 
and moreover $r_1,s_1,r_2',s_2'\in[r',s']$. 
Therefore we can apply Theorem~\ref{wigner_cont_th} to obtain 
\begin{equation}\label{C2_eq1}
\Vert\Wig(f_2,f_1)\Vert_{M^{r',s'}_{\Wig(\phi_1,\phi_2)}(\pi^\#,\theta\times\theta)}
\le
\Vert f_1\Vert_{M^{r_1,s_1}_{\phi_1}(\pi,\theta)}\cdot\Vert f_2\Vert_{M^{r_2',s_2'}_{\phi_2}(\pi,\theta)} 
\end{equation}
whenever $f_1,f_2\in\Hc_{-\infty}$. 

On the other hand, 
if $a\in M^{r,s}_{\Wig(\phi_1,\phi_2)}(\pi^\#,\theta\times\theta)$, then 
$$\begin{aligned}
(\Op^\theta(a)f_1\mid f_2)
&=(a\mid\Wig(f_2,f_1))_{L^2(\Xi^*)} \\
&=(\Ac^{\pi^\#,\theta\times\theta}_{\Wig(\phi_1,\phi_2)} a
\mid\Ac^{\pi^\#,\theta\times\theta}_{\Wig(\phi_1,\phi_2)}(\Wig(f_2,f_1)))_{L^2(\Xi\times\Xi)}
\end{aligned}$$
where the first equality follows by \eqref{loc1_eq2}, 
while the second equality can be proved by using Lemma~\ref{loc5}\eqref{loc5_item2}.
Then H\"older's inequality for mixed-norm spaces 
(see for instance Lemma~11.1.2(b) in \cite{Gr01}) 
shows that 
$$\begin{aligned}
\vert(\Op^\theta(a)f_1\mid f_2)\vert
&\le 
\Vert\Ac^{\pi^\#,\theta\times\theta}_{\Wig(\phi_1,\phi_2)} a\Vert_{L^{r,s}(\Xi\times\Xi)} \cdot \Vert\Ac^{\pi^\#,\theta\times\theta}_{\Wig(\phi_1,\phi_2)}(\Wig(f_2,f_1))\Vert_{L^{r',s'}(\Xi\times\Xi)}\\ 
&=\Vert a\Vert_{M_{\Wig(\phi_1,\phi_2)}^{r,s}(\pi^\#,\theta\times\theta)}\cdot 
\Vert\Wig(f_2,f_1)\Vert_{M_{\Wig(\phi_1,\phi_2)}^{r',s'}(\pi^\#,\theta\times\theta)} \\
&\le\Vert a\Vert_{M_{\Wig(\phi_1,\phi_2)}^{r,s}(\pi^\#,\theta\times\theta)} \cdot
\Vert f_1\Vert_{M^{r_1,s_1}_{\phi_1}(\pi,\theta)}\cdot\Vert f_2\Vert_{M^{r_2',s_2'}_{\phi_2}(\pi,\theta)},
\end{aligned}$$
where the latter inequality follows by \eqref{C2_eq1}. 
Now the assertion follows by a straightforward argument 
that uses the duality of the mixed-norm spaces 
(see Lemma~11.1.2(d) in \cite{Gr01}). 
\end{proof}

\begin{corollary}\label{C3}
Let $\phi_1,\phi_2\in\Hc_\infty\setminus\{0\}$, 
and assume the following hypotheses: 
\begin{enumerate}
\item 
The representation $\pi$ satisfies both the density condition and the orthogonality relations  
along the linear mapping $\theta\colon\Xi\to\mg$. 
\item The localized Weyl calculus for the representation $\pi$ along $\theta$ is regular. 
\end{enumerate}
Then for every $a\in M^{\infty,1}_{\Wig(\phi_1,\phi_2)}(\pi^{\#})$ we have 
$\Op^\theta(a)\in\Bc(\Hc)$, and  
the linear mapping 
$\Op^\theta\colon M^{\infty,1}_{\Wig(\phi_1,\phi_2)}(\pi^\#,\theta\times\theta)\to\Bc(\Hc)$ is continuous. 
\end{corollary}

\begin{proof}
This is the special case of Corollary~\ref{C2} with with 
$r_1=s_1=r_2=s_2=2$, $r=1$, and $s=\infty$, since  
Remark~\ref{mod_L2} shows that $M^{2,2}_{\phi_j}(\pi,\theta)=\Hc$ 
for $j=1,2$. 
\end{proof}

\subsection*{Trace-class operators obtained by localized Weyl calculus}

\begin{lemma}\label{wavelets}
Let the representation $\pi\colon M\to\Bc(\Hc)$ satisfy the orthogonality relations along the linear mapping $\theta\colon\Xi\to\mg$, 
and pick $\phi_0\in\Hc_\infty$ with $\Vert\phi_0\Vert=1$. 
Then the following assertions hold: 
\begin{enumerate}
\item\label{wavelets_item1} 
The operator $\Ac^{\pi,\theta}_{\phi_0}\colon \Hc\to L^2(\Xi)$, 
$f\mapsto \Ac^{\pi,\theta}_{\phi_0} f$, 
is an isometry whose image is the reproducing kernel Hilbert space 
associated with the reproducing kernel 
$$K\colon\Xi\times\Xi\to\CC,\quad 
K(X_1,X_2)
=(\pi(\exp_M(\theta(X_1)))\phi_0\mid\pi(\exp_M(\theta(X_2)))\phi_0).$$
The orthogonal projection from $L^2(\Xi)$ onto $\Ran\Ac^{\pi,\theta}_{\phi_0}$ 
is just the integral operator defined by the integral kernel $K$. 
\item\label{wavelets_item2}
For every $\phi,f\in\Hc$ we have 
\begin{equation*}
\int\limits_{\Xi}(\Ac^{\pi,\theta}_{\phi_0}f)(X)\cdot
\pi(\exp_M(\theta(X)))\phi\,\de X
=(\phi\mid\phi_0)f.
\end{equation*}
In particular, for every $f\in\Hc$ we have  
\begin{equation}\label{wavelets_eq1}
\int\limits_{\Xi}(\Ac^{\pi,\theta}_{\phi_0} f)(X)\cdot
\pi(\exp_M(\theta(X)))\phi_0\,\de X=f,
\end{equation}
where the integral is weakly convergent in $\Hc$. 
\newcounter{enumi_saved}
\setcounter{enumi_saved}{\value{enumi}}
\end{enumerate}
Assume that  the representation $\pi$ satisfies the growth condition along $\theta$.
Also,  assume that for every $u \in \U(\mg_{\CC})$ the function $\Vert \de \pi (u) \pi(\exp_M(\theta(\cdot)))\phi_0 \Vert$ has polynomial growth, 
then moreover we have:
\begin{enumerate}
\setcounter{enumi}{\value{enumi_saved}}
\item\label{wavelets_item3} 
If $f\in\Hc_\infty$, then the integral in \eqref{wavelets_eq1} is convergent with respect to the topology of $\Hc_\infty$. 
\item\label{wavelets_item4} If $f\in\Hc_{-\infty}$, then \eqref{wavelets_eq1} holds with the integral  convergent in the $w^*$-topology. 
\item\label{wavelets_item5} 
We have 
$\Hc_\infty=\{f\in\Hc_{-\infty}\mid \Ac^{\pi,\theta}_{\phi_0} f\in\Sc(\Xi)\}$. 
\end{enumerate}
\end{lemma}

\begin{proof}
Assertion~\eqref{wavelets_item1} follows at once by the orthogonality relations 
along with \cite[Prop.~2.12]{Fue05}. 
Then Assertion~\eqref{wavelets_item2} follows by an application of \cite[Prop.~2.11]{Fue05}. 
The proof for Assertions \eqref{wavelets_item3}--\eqref{wavelets_item5} can be supplied by adapting the method of proof of \cite[Cor.~2.9]{BB09c}. 
We omit the details. 
\end{proof}

\begin{remark}\label{ad-remark}
\normalfont
We note here that in the setting of Lemma~\ref{wavelets}, the condition that 
for all $u\in\U(\mg_{\CC})$ and $\phi\in\Hc_\infty$ 
the function $\Vert\de\pi(\Ad_{\U(\mg_{\CC})}(\exp_M(\theta(\cdot)))u)\phi\Vert$ has polynomial growth on $\Xi$ implies that for all
$f\in\Hc_{-\infty}$, $\phi\in\Hc_\infty$,
 the function $
\Ac^{\pi,\theta}_\phi f$ has polynomial growth as well.

In fact, if $f\in\Hc_{-\infty}$, then there exists $u\in\U(\mg_{\CC})$ such that 
for every $\psi\in\Hc_\infty$ we have $\vert(f\mid\psi)\vert\le\Vert\de\pi(u)\psi\Vert$. 
(See Remark~\ref{smooth_top}.)
Then we have 
$$\begin{aligned}
\vert(\Ac^{\pi,\theta}_\phi f)(\cdot)\vert
&=\vert(f\mid\pi(\exp_M(\theta(\cdot)))\phi)\vert
\le\Vert\de\pi(u)\pi(\exp_M(\theta(\cdot))\phi)\Vert \\
&=\Vert\de\pi(\Ad_{\U(\mg_{\CC})}(\exp_M(\theta(\cdot)))u)\phi\Vert
\end{aligned} $$
and the latter function has polynomial growth by assumption. 
\qed
\end{remark}

By using the method of proof of \cite[Prop.~2.27]{BB09c} 
we can now obtain the following sufficient condition for a symbol to give rise to a trace-class operator. 

\begin{proposition}\label{trace_class}
Let $\phi_1,\phi_2\in\Hc_\infty$ such that $\Vert\phi_j\Vert=1$ 
and for every $u \in \U(\mg_{\CC})$ the function 
$\Vert \de \pi (u) \pi(\exp_M(\theta(\cdot)))\phi_j \Vert$
has polynomial growth, for $j=1, 2$,  and assume the following hypotheses: 
\begin{enumerate}
\item 
The representation $\pi$ satisfies both the density condition and the orthogonality relations  
along the linear mapping $\theta\colon\Xi\to\mg$. 
\item The localized Weyl calculus for the representation $\pi$ along $\theta$ is regular.  
\end{enumerate}
Then for every $a\in M^{1,1}_{\Wig(\phi_1,\phi_2)}(\pi^{\#},\theta\times\theta)$ we have 
$\Op^\theta(a)\in\Sg_1(\Hc)$, and  
the linear mapping 
$\Op^\theta\colon M^{1,1}_{\Wig(\phi_1,\phi_2)}(\pi^\#,\theta\times\theta)\to\Sg_1(\Hc)$ is continuous. 
\end{proposition}

\begin{proof}
It follows by Lemma~\ref{loc5}\eqref{loc5_item2}, Lemma~\ref{loc55}\eqref{loc55_item2} and 
Remark~\ref{mod_equiv} that the representation  
$\pi^\#\colon M\ltimes M\to\Bc(L^2(\Xi^*))$ satisfies both the orthogonality relations and the growth condition along the linear mapping 
$\theta\times\theta\colon\Xi\times\Xi\to\mg\ltimes\mg$. 
Moreover, it is easily seen that the function 
$\Phi_0:=\Wig(\phi_1,\phi_2)\in\Sc(\Xi^*)$ has the property that  for every $u\in \U((\mg\ltimes \mg)_{\CC}))
$ the norm of 
$\de\pi^\#(u)\pi^\#(\exp_{M\ltimes M}((\theta\times\theta)(\cdot))))\Phi_0$ has polynomial growth on 
$\Xi\times \Xi$, 
since a similar property has the rank-one operator $\Op^\theta(\Phi_0)=(\cdot\mid\phi_2)\phi_1\in\Sg_2(\Hc)$ with respect to the representation $\pi^\ltimes$, 
as a direct consequence of the calculation~\eqref{trace_class_eq2} below. 
Therefore we can use Lemma~\ref{wavelets}\eqref{wavelets_item4} 
for the representation $\pi^\#$ 
to see that for arbitrary $a\in\Sc'(\Xi^*)$ we have 
$$a=\iint\limits_{\Xi\times\Xi}
(\Ac^{\pi^\#,\theta\times\theta}_{\Phi_0}a)(X, Y)\cdot\pi^{\#}(\exp_{M\ltimes M}(\theta(X),\theta( Y)))\Phi_0
\de X\de  Y, $$
whence by \eqref{loc1_eq2_bis} we get 
\begin{equation}\label{trace_class_eq1}
\Op^\pi(a)=\iint\limits_{\Xi\times\Xi}
(\Ac^{\pi^\#,\theta\times\theta}_{\Phi_0}a)(X, Y)\cdot
\Op^\theta(\pi^{\#}(\exp_{M\ltimes M}(\theta(X),\theta( Y)))\Phi_0)
\de X\de  Y
\end{equation}
where the latter integral is weakly convergent in $\Lc(\Hc_\infty,\Hc_{-\infty})$ 
($\simeq\Lc(\Hc_{-\infty},\Hc_\infty)'$ by Proposition~\ref{twice}\eqref{twice_item2}). 
On the other hand, for arbitrary $X, Y\in\Xi$ we get by 
Remarks \ref{moyal_rem} and \ref{extension}
\begin{align}
\Op^\theta(\pi^{\#}(\exp_{M\ltimes M} & (\theta(X),\theta( Y)))\Phi_0) \nonumber\\
&=\pi(\exp_M(\theta(X)+\theta( Y)))\circ
\Op^\theta(\Phi_0)\circ\pi(\exp_M (\theta(X)))^{-1} \nonumber\\
&=(\cdot\mid\pi(\exp_M(\theta(X)))\phi_2)\pi(\exp_M(\theta(X+ Y)))\phi_1.\label{trace_class_eq2} 
\end{align}
In particular, 
$\Op^\theta(\pi^{\#}(\exp_{M\ltimes M}(\theta(X),\theta( Y)))\Phi_0)
\in\Sg_1(\Hc)$ 
and 
$$\begin{aligned}
\Vert
\Op^\theta(\pi^{\#}(\exp_{M\ltimes M}&(\theta(X),\theta( Y)))\Phi_0)
\Vert_1 \\
& =\Vert\pi(\exp_M(\theta(X+ Y)))\phi_1\Vert
\cdot\Vert\pi(\exp_M(\theta(X)))\phi_2\Vert \\
&=1.
\end{aligned}$$ 
It then follows that the integral in \eqref{trace_class_eq1} 
is absolutely convergent in $\Sg_1(\Hc)$ for 
$a\in M^{1,1}_{\Phi_0}(\pi^{\#},\theta\times\theta)$ 
and moreover we have 
$$\Vert\Op^\theta(a)\Vert_1\le
\iint\limits_{\Xi\times\Xi}
\vert(\Ac^{\pi^\#,\theta\times\theta}_{\Phi_0}a)(X, Y)\vert
\de X\de  Y
=\Vert a\Vert_{M^{1,1}_\Phi(\pi^{\#},\theta\times\theta)}$$
which concludes the proof. 
\end{proof}

\section{Applications to the magnetic Weyl calculus}

We proved in \cite{BB09a} that the magnetic Weyl calculus on $\RR^n$ constructed in \cite{MP04} 
can be alternatively described as the localized Weyl calculus for a suitable representation. 
This point of view actually allowed us to construct magnetic Weyl calculi 
on any simply connected nilpotent Lie group $G$, by using an appropriate representation 
$\pi\colon M=\Fc\rtimes G\to\Bc(L^2(G))$ and linear mappings $\theta^A\colon\gg\times\gg^*\to\mg$. 

We shall see in the present section that all of the conditions studied in Sections \ref{Sect2}~and~\ref{localized}
are met by $\pi$ and $\theta^A$ (see Corollary~\ref{conditions} below), 
provided the coefficients of the magnetic potential $A\in\Omega^1(G)$ have polynomial growth. 
Therefore, the abstract results of the previous sections can be used for obtaining continuity and nuclearity properties for the magnetic Weyl calculus (see Corollaries \ref{C2_mag}--\ref{C4_mag} below).
\begin{notation}
\normalfont
For any Lie group $G$ we denote by 
$\lambda\colon G\to\End(\Ci(G))$, $g\mapsto\lambda_g$,  
the left regular representation defined by $(\lambda_g\phi)(x)=\phi(g^{-1}x)$ 
for every $x,g\in G$ and $\phi\in\Ci(G)$. 
Moreover, we denote by $\1$ the constant function which is identically equal to~1 on $G$. 
(This should not be confused with the unit element of $G$, which is denoted in the same way.)
\qed
\end{notation}

We now recall the following notion from \cite{BB09a}. 

\begin{definition}\label{orbit0}
\normalfont
Let $G$ be a finite-dimensional Lie group.  
A linear space $\Fc$ of real functions on $G$ is said to be \emph{admissible}
if it is endowed with a sequentially complete, locally convex topology and satisfies the following conditions: 
\begin{enumerate}
\item\label{orbit0_item1}
The linear space $\Fc$ is invariant under the representation of $G$ by left translations, 
that is, if $\phi\in\Fc$ and $g\in G$ then $\lambda_g\phi\in\Fc$. 
\item\label{orbit0_item2}
We have a continuous inclusion mapping $\Fc\hookrightarrow\Ci(G)$. 
\item\label{orbit0_item3}
The mapping $G\times\Fc\to\Fc$, $(g,\phi)\mapsto\lambda_g\phi$ is smooth. 
For every $\phi\in\Fc$ we denote by $\dot{\lambda}(\cdot)\phi\colon\gg\to\Fc$ 
the differential of the mapping $g\mapsto\lambda_g\phi$ at the point $\1\in G$. 
\item\label{orbit0_item4} 
For every $g_1,g_2\in G$ with $g_1\ne g_2$ there exists $\phi\in\Fc$ with $\phi(g_1)\ne\phi(g_2)$. 
\item\label{orbit0_item5} 
We have $\{\phi'_g\mid\phi\in\Fc\}=T_g^*G$ for every $g\in G$. 
\end{enumerate}
For instance, the function space $\Ci_{\RR}(G)$ is admissible. 
\qed
\end{definition}

\begin{proposition}\label{F_0}
Let $G$ be a finite-dimensional simply connected nilpotent Lie group 
with the inverse of the exponential map denoted by $\log_G\colon G\to\gg$. 
If we define
\begin{equation}\label{F_0_eq1}
\Fc_G:=\spa_{\RR}(\{\lambda_g(\xi\circ\log_G)\mid\xi\in\gg^*,g\in G\}),
\end{equation}
then the following assertions hold: 
\begin{enumerate}
\item\label{F_0_item1} 
$\Fc_G$ is a finite dimensional linear subspace of $\Ci(G)$ 
which is invariant under the left regular representation 
and contains the constant functions. 
\item\label{F_0_item2} 
The semi-direct product $M_0:=\Fc_G\rtimes_\lambda G$ is a finite-dimensional simply connected nilpotent Lie group. 
\end{enumerate}
\end{proposition}

\begin{proof}
Since $G$ is a simply connected nilpotent Lie group, 
we may assume that $G=(\gg,\ast)$. 

\eqref{F_0_item1}
It is clear that the linear space $\Fc_G$ is invariant under the left regular representation. 
On the other hand, for every $V,X\in\gg$ and $\xi\in\gg^*$ we have 
$$(\lambda_V\xi)(X)=\langle\xi,(-V)\ast X\rangle
=\langle\xi,-V+X+\frac{1}{2}[-V,X]+\cdots\rangle. $$
Thus, if we denote by $N$ the nilpotency index of $\gg$, 
then we see that $\Fc_G$ consists of polynomial functions on $\gg$ of degree $\le N$, hence $\dim\Fc_G<\infty$. 
Moreover, if $\zg$ denotes the center of $\gg$ and we pick 
$V\in\zg$ and $\xi\in\gg^*$, 
then $\lambda_V\xi=-\langle\xi,V\rangle\1+\xi$.  
We thus see that the constant functions belong to $\Fc_G$. 

\eqref{F_0_item2} 
On the Lie algebra level we have 
$\mg_0:=\Fc_G\rtimes_{\dot\lambda}\gg$, 
and both $\Fc_G$ and $\gg$ are nilpotent Lie algebras. 
Therefore Engel's theorem shows that, for proving that $\mg_0$ is nilpotent, 
it is enough to check that the adjoint action $\ad_{\mg_0}$ gives a 
representation of $\gg$ on $\Fc_G$ by \emph{nilpotent} endomorphisms. 
This representation is just $\dot\lambda\colon\gg\to\End(\Fc_G)$ 
hence, by the theorem on weight space decompositions 
for representations of nilpotent Lie algebras 
(see for instance \cite[Th.~2.9]{Ca05}), 
it suffices to prove the following fact: 
\emph{If $\alpha\in\gg^*$, $\phi\in\Fc_G\setminus\{0\}$, 
and for every $X\in\gg$ we have $\dot\lambda(X)\phi=\alpha(X)\phi$, 
then $\alpha=0$.} 

To this end, let $X_0\in\gg$ arbitrary. 
Since $\dot\lambda(X_0)\phi=\alpha(X_0)\phi$, 
it follows that for every $Y\in\gg$ and $t\in\RR$ we have $\phi((-tX_0)\ast Y)=\ee^{t\alpha(X_0)}\phi(Y)$. 
We have seen above that $\Fc_G$ consists of polynomial functions on $\gg$ of degree $\le N$, 
therefore for every $Y\in\gg$ there exists a constant $C_{\phi,Y}>0$ such that 
$$(\forall t\in\RR)\quad 
\ee^{t\alpha(X_0)}\vert\phi(Y)\vert
=\vert\phi((-tX_0)\ast Y)\vert
\le C_{\phi,Y}(1+\vert t\vert)^{N^2}. $$
On the other hand, since $\phi\in\Fc_G\setminus\{0\}$, there exists $Y\in\gg$ such that $\phi(Y)\ne0$, 
and then the above inequality shows that $\alpha(X_0)=0$. 
This holds for arbitrary $X_0\in\gg$, 
hence $\alpha=0$, 
as we wished for. 
\end{proof}

\begin{theorem}\label{smooth_vect}
Let $G$ be a finite-dimensional simply connected nilpotent Lie group 
with an admissible function space $\Fc$ such that 
there exist the continuous inclusion maps 
$\gg^*\hookrightarrow\Fc\hookrightarrow\Cpol(G)$, 
where the embedding $\gg^*\hookrightarrow\Fc$ is given by $\xi\mapsto\xi\circ\log_G$. 
Denote $M=\Fc\rtimes_\lambda G$, fix $\epsilon\in\RR\setminus\{0\}$, and 
consider the unitary representation 
$\pi\colon M\to\Bc(L^2(G))$, $\pi(\phi,g)f=\ee^{\ie\epsilon\phi}\lambda_g f$ 
for all $\phi\in\Fc$, $g\in G$, and $f\in L^2(G)$. 
Then $\pi$ is a nuclearly smooth representation and its space of smooth vectors 
is the Schwartz space $\Sc(G)$. 
\end{theorem}

\begin{proof}
Let us denote $\Hc=L^2(G)$ and let $\Hc_\infty$ be the space of smooth vectors for the representation $\pi$. 
We first check that $\Sc(G)=\Hc_\infty$. 

For proving that $\Sc(G)\subseteq\Hc_\infty$, 
let $f\in\Sc(G)$ arbitrary. 
Since $\Fc\hookrightarrow\Cpol(G)$, 
it follows at once that for every $\phi\in\Fc$ and $g\in G$ we have 
$\pi(\phi,\cdot)f\in\Ci(G,\Hc)$ and $\pi(\cdot,g)f\in\Ci(\Fc,\Hc)$. 
It then follows by \cite[Sect.~I]{Ne01} 
(see also \cite[Th.~3.4.3]{Ha82}) 
that $\pi(\cdot)f\in\Ci(M,\Hc)$, hence $f\in\Hc_\infty$. 

To prove the converse inclusion 
$\Sc(G)\subseteq\Hc_\infty$ we need the function space 
$\Fc_G$ defined in \eqref{F_0_eq1}. 
Since $\Fc$ contains $\{\xi\circ\log_G\mid\xi\in\gg^*\}$ and is invariant under the left regular representation of $G$, 
we get $\Fc_G\hookrightarrow\Fc$. 
Now Proposition~\ref{F_0} shows that $M_0:=\Fc_G\rtimes G$ is a finite-dimensional nilpotent Lie group. 
Since $\gg^*\hookrightarrow\Fc_G$, it is easily seen that the unitary representation $\pi_0:=\pi\vert_{M_0}\colon M_0\to\Bc(\Hc)$ is irreducible. 
Let $\Hc_{\infty,\pi_0}$ be its space of smooth vectors. 
If $\delta_{\1}\colon\Ci(G)\to\CC$ is the Dirac distribution 
at $\1\in G$, 
then the discussion in \cite[subsect.~2.4]{BB09a} 
shows that $\Fc_G\times\{0\}$ is a polarization 
for the functional $(\delta_{\1}\vert_{\Fc_G},0)\in\mg_0^*$, 
and the corresponding induced representation is just~$\pi_0$. 
Now $\Hc_{\infty,\pi_0}=\Sc(G)$ by \cite[Cor.~to Th.~3.1]{CGP77}. 
Therefore we get the continuous inclusion $\Hc_\infty\hookrightarrow\Sc(G)$, 
which completes the proof for the equality $\Sc(G)=\Hc_\infty$. 

Furthermore, it easily follows by \cite[Cor.~A.2.4]{CG90} that 
$\Hc_\infty=\Sc(G)=\Sc(\gg)$ as locally convex spaces. 
On the other hand, it is well known that $\Sc(\gg)$ 
is a nuclear Fr\'echet space; see for instance \cite{Tr67}. 
Finally, both mappings 
$M\times\Sc(G)\to\Sc(G)$, $(m,\phi)\mapsto\pi(m)\phi$, 
and $\mg\times\Sc(G)\to\Sc(G)$, $(X,\phi)\mapsto\de\pi(X)\phi$ 
are continuous as a direct consequence of \cite[Th.~A.2.6]{CG90}, 
and this concludes the proof of the fact that $\pi$ is a nuclearly smooth representation. 
\end{proof}

We now prove that the conclusion of Theorem~\ref{smooth_vect} 
actually holds under a much stronger form. 

\begin{corollary}\label{twice_smooth}
In the setting of Theorem~\ref{smooth_vect}, 
the unitary representation $\pi$ is twice nuclearly smooth. 
\end{corollary}

\begin{proof}
The proof has two stages. 
For the sake of simplicity we assume $\epsilon=1$, 
however it is clear that the following reasonings carry over to the general case. 

$1^\circ$ 
We first make the following remark: 
For $j=1,2$, let $G_j$ 
be a finite-dimensional simply connected nilpotent Lie group 
with an admissible function space $\Fc_j$ such that 
$\gg_j^*\hookrightarrow\Fc_j\hookrightarrow\Cpol(G_j)$ 
as in Theorem~\ref{smooth_vect}. 
Also define the group $M_j=\Fc_j\rtimes_\lambda G_j$ and 
the unitary representation 
$\pi_j\colon M_j\to\Bc(L^2(G_j))$, $\pi_j(\phi,g)f=\ee^{\ie(-1)^{j-1}\phi}\lambda_g f$ 
for all $\phi\in\Fc_j$, $g\in G_j$, and $f\in L^2(G_j)$.  
Now consider the direct product group $G_0:=G_1\times G_2$,  
the function space 
$$\Fc_0:=(\Fc_1\otimes\1)+(\1\otimes\Fc_2)\hookrightarrow\Cpol(G_0),$$ 
and the representation 
$\pi_0\colon M_0\to\Bc(L^2(G_0))$, $\pi_0(\phi,g)f=\ee^{\ie\phi}\lambda_g f$ 
for all $\phi\in\Fc_0$, $g\in G_0$, and $f\in L^2(G_0)$, 
where $M_0:=\Fc_0\rtimes_\lambda G_0$. 
Then $\Fc_0$ is an admissible function space on $G_0$ 
and there exists a 1-dimensional central subgroup $N\subseteq M_1\times M_2$ 
such that $N\subseteq\Ker(\pi_1\otimes\pi_2)$,  
and we have $M_0=(M_1\times M_2)/N$.  
Moreover, the representation $\pi_0$ is equal to 
$\pi_1\otimes\pi_2$ factorized modulo~$N$. 

In fact, let us define the linear map
$$\Delta\colon\Fc_1\times\Fc_2\to\Fc_0,\quad 
(\phi_1,\phi_2)\mapsto\phi_1\otimes\1-\1\otimes\phi_2.$$ 
Then $\Ran\Delta=\Fc_0$ and $\Ker\Delta=\{(t\1,t\1)\mid t\in\RR\}\simeq\RR$, 
hence we get a linear isomorphism $\Fc_0\simeq(\Fc_1\times\Fc_2)/\Ker\Delta$, 
and this can be used to define the topology of~$\Fc_0$. 
Moreover, it is clear that $\Ker\Delta$ is contained in the center 
of $\mg_1\times\mg_2\simeq\mg_0$  and $\Ker\Delta\subseteq\Ker(\de(\pi_1\otimes\pi_2))$, 
hence the above remark holds for $N=\exp_{M_0}(\Ker\Delta)$. 

$2^\circ$ 
We now come back to the proof of the corollary. 
We already know from Theorem~\ref{smooth_vect} that the representation $\pi$ is nuclearly smooth. 
Moreover, by using the remark of stage~$1^\circ$ for $G_1=G_2=G$ 
along with Theorem~\ref{smooth_vect} for the group $G\times G$, 
we easily see that the space of smooth vectors for the representation $\pi\otimes\bar\pi$ 
is linear and topologically isomorphic to $\Sc(G\times G)$, 
which in turn is isomorphic to $\Sc(G)\hotimes\Sc(G)$ 
(see for instance \cite{Tr67}). 
On the other hand, $\Sc(G)$ is the space of smooth vectors for $\pi$, 
by Theorem~\ref{smooth_vect}. 
Thus the representation $\pi$ also satisfies the second condition in the definition of a twice nuclearly smooth representation 
(see Definition~\ref{smoothness}), and we are done. 
\end{proof}

\begin{notation}\label{right_invar}
\normalfont
Let $G$ be any Lie group with the Lie algebra $\gg$ and with the 
space of globally defined smooth vector fields 
(that is, global sections in its tangent bundle) denoted by $\Xg(G)$ 
and the space of globally defined smooth 1-forms  
(that is, global sections in its cotangent bundle) denoted by $\Omega^1(G)$. 
Then there exists a natural bilinear map 
$$\langle\cdot,\cdot\rangle\colon\Omega^1(G)\times\Xg(G)\to\Ci(G)$$
defined as usually by evaluations at every point of~$G$. 

Moreover, for arbitrary $g\in G$, we denote the corresponding right-translation mapping by 
$R_g\colon G\to G$, $h\mapsto hg$. 
Then we define the injective linear mapping 
$$\iotaR\colon\gg\to\Xg(G)$$ 
by $(\iotaR X)(g)=(T_{\1}(R_g))X\in T_g G$ 
for all $g\in G$ and $X\in\gg$. 
\qed
\end{notation}

\begin{corollary}\label{conditions}
Assume the setting of Theorem~\ref{smooth_vect}.  
If we have $A\in\Omega^1(G)$ such that   
$\langle A,\iotaR X\rangle\in\Fc$ whenever $X\in\gg$,  
then we define the linear mapping 
$$\theta^A\colon\gg\times\gg^*\to\mg=\Fc\ltimes_{\dot\lambda}\gg,\quad 
(X,\xi)\mapsto(\xi\circ\log_G+\langle A,\iotaR X\rangle,X).$$ 
Then for every $\epsilon\in\RR\setminus\{0\}$ 
the representation $\pi_\epsilon\colon M\to\Bc(L^2(G))$ 
has the following properties: 
\begin{enumerate}
\item\label{conditions_item1} 
The representation $\pi_\epsilon$ satisfies the orthogonality relations along the mapping~$\theta^A$. 
\item\label{conditions_item2} 
The representation $\pi_\epsilon$ satisfies the growth condition  along~$\theta^A$.
\item\label{conditions_item3} 
The localized Weyl calculus for $\pi_\epsilon$ along $\theta^A$ is regular and defines a unitary operator 
$\Op^{\theta^A}\colon L^2(\gg\times\gg^*)\to\Sg_2(L^2(G))$.
\item\label{conditions_item4} 
If $u\in\U(\mg_{\CC})$ and $\phi\in\Sc(G)$,  
the function $\Vert\de\pi(\Ad_{\U(\mg_{\CC})}(\exp_M(\theta^A(\cdot)))u)\phi\Vert$ has polynomial growth on~$\gg\times\gg^*$.
\end{enumerate} 
\end{corollary}

\begin{proof} 
Throughout the proof we assume $\epsilon=1$ 
and we denote $\pi_1=\pi$ for the sake of simplicity. 
The case of an arbitrary $\epsilon\in\RR\setminus\{0\}$ can be handled by a similar method.
Since $G$ is simply connected, we may assume 
$G=(\gg,\ast)$. 
Then the space of smooth vectors for $\pi_\epsilon$ is equal to $\Sc(\gg)$ by Theorem~\ref{smooth_vect}.

\eqref{conditions_item1} 
The assertion follows by \cite[Th.~2.8(1)]{BB10a}. 
 
\eqref{conditions_item2}
To check the growth condition \eqref{wigner_def_eq0}
we shall denote for every $X\in\gg$, 
$$\Psi_X\colon\gg\to\gg, \quad \Psi_X(Y)=\int\limits_0^1Y\ast(sX)\de s$$
and also 
$$\tau_A(X,Y)=\exp\Bigl(\ie\int\limits_0^1
\langle A,\iotaR X\rangle((-sX)\ast Y)\de s\Bigr)$$
for $X,Y\in\gg$. 
It then follows by \cite[Prop.~2.9(1)]{BB10a} that for every $f,\phi\in\Sc(\gg)$ we have 
$$(\Ac^{\pi,\theta^A}_\phi f)(X,\xi)=\hskip-2pt\int\limits_{\gg} 
 \ee^{\ie\hake{\xi,Y}}\overline{\tau_A(X,-\Psi_X^{-1}(Y))} 
f(-\Psi_X^{-1}(Y))\overline{\phi((-X)\ast (-\Psi_X^{-1}(Y)))}\,\de Y. 
$$
Therefore the function $\Ac^{\pi,\theta^A}_\phi f\colon\gg\times\gg^*\to\CC$ 
is a partial inverse Fourier transform of the function defined on 
$\gg\times\gg$ by 
$$(X,Y)\mapsto\overline{\tau_A(X,-\Psi_X^{-1}(Y))}f(-\Psi_X^{-1}(Y))\overline{\phi((-X)\ast (-\Psi_X^{-1}(Y)))}\colon\gg\to\CC.$$
On the other hand, it was noted in the proof of \cite[Th.~4.4(4)]{BB09a} 
that each of the mappings 
$\Sigma_1,\Sigma_2\colon\gg\times\gg\to\gg\times\gg$ are defined by 
$$
\Sigma_1(Y,Z)=(-Y,Y\ast(-Z))
\quad\text{and}\quad 
\Sigma_2(V,W)=(-\Psi_W(V),W).
$$ 
is a polynomial diffeomorphisms whose inverse is a polynomial. 
Since 
$$\Sigma_2^{-1}(Y,X)=(\Psi_X^{-1}(-Y),X)$$ 
and 
$\tau_A\in\Cpol(\gg\times\gg)$, it then easily follows by \cite[Lemma~A.2.1(a)]{CG90} 
that we have a well-defined continuous sesquilinear mapping 
$$\Sc(\gg)\times\Sc(\gg)\to\Sc(\gg\times\gg^*),\quad 
(f,\phi)\mapsto\Ac^{\pi,\theta^A}_\phi f. $$
Thus the representation $\pi$ satisfies the growth condition along the mapping~$\theta^A$.

\eqref{conditions_item3}
Use the above Assertion~\eqref{conditions_item3} along with~\cite[Th.~4.4(4)]{BB09a}. 

\eqref{conditions_item4}
The assertion follows as a direct consequence of \cite[Lemma~2.5]{BB10a}.
\end{proof}


In the next corollaries we denote by $\pi$ the representation 
$\pi_\epsilon$ in Theorem~\ref{smooth_vect} for $\epsilon=1$. 
Recall that we work with a finite-dimensional simply connected nilpotent Lie group $G$ with an admissible function space $\Fc$ such that 
there exist the continuous inclusion maps 
$\gg^*\hookrightarrow\Fc\hookrightarrow\Cpol(G)$, 
where the embedding $\gg^*\hookrightarrow\Fc$ is given by $\xi\mapsto\xi\circ\log_G$. 
Moreover $M=\Fc\rtimes_\lambda G$, and the aforementioned unitary representation 
$\pi\colon M\to\Bc(L^2(G))$ is defined by 
$\pi(\phi,g)f=\ee^{\ie\phi}\lambda_g f$ 
for all $\phi\in\Fc$, $g\in G$, and $f\in L^2(G)$. 

If we have $A\in\Omega^1(G)$ such that   
$\langle A,\iotaR X\rangle\in\Fc$ whenever $X\in\gg$,  
and we define the linear mapping 
$$\theta^A\colon\gg\times\gg^*\to\mg=\Fc\ltimes_{\dot\lambda}\gg,\quad 
(X,\xi)\mapsto(\xi\circ\log_G+\langle A,\iotaR X\rangle,X)$$
as in Corollary~\ref{conditions}, then one can consider 
the modulation spaces of symbols for the localized Weyl calculus 
for the representation $\pi$ along the linear mapping $\theta^A$. 
These are just the modulation spaces for the representation 
$\pi^\#\colon M\ltimes M\to\Bc(L^2(\gg\times\gg^*))$ 
with respect to the linear mapping $(\theta^A,\theta^A)\colon (\gg\times\gg^*)\times(\gg\times\gg^*) 
\to \mg\ltimes\mg$. 
It follows by Remark~\ref{moyal_rem} that for arbitrary $\Phi\in\Sc(\gg\times\gg^*)$ and $F\in\Sc'(\gg\times\gg^*)$ the corresponding ambiguity function 
$\Ac^{\pi^\#,\theta^A\times\theta^A}_\Phi F\colon(\gg\times\gg^*)\times(\gg\times\gg^*)\to\CC$
is given by the formula 
$$\begin{aligned}
(\Ac^{\pi^\#,\theta^A\times\theta^A}_\Phi F) & ((X_1,\xi_1),(X_2,\xi_2)) \\
&=(\pi^\#(\exp_{M\ltimes M}(\theta^A(X_1,\xi_1),\theta^A(X_2,\xi_2))F
\mid\Phi)_{L^2(\gg\times\gg^*)} \\
&=
\iint\limits_{\gg\times\gg^*}
(\ee^{\ie\langle\cdot,(X_1+X_2,\xi_1+\xi_2)\rangle}\#^{\theta^A} F\#^{\theta^A} \ee^{-\ie\langle\cdot,(X_1,\xi_1)\rangle})\overline{\Phi(\cdot)} 
\end{aligned}
$$
where $\#^{\theta^A}$ stands for the Moyal product on $\gg\times\gg^*$ 
defined by means of the magnetic potential $A$. 
For $r,s\in[1,\infty]$ and the window function $\Phi\in\Sc(\gg\times\gg^*)$ 
we have the modulation space of symbols 
$$M^{r,s}_\Phi(\pi^\#,\theta^A\times\theta^A)
=\{F\in\Sc'(\gg\times\gg^*)\mid \Ac^{\pi^\#,\theta^A\times\theta^A}_\Phi F
\in L^{r,s}((\gg\times\gg^*)\times(\gg\times\gg^*))\}.$$ 

\begin{corollary}\label{C2_mag}
In the above setting, 
pick $\phi_1,\phi_2\in\Sc(G)\setminus\{0\}$.   
If $r,s,r_1,s_1,r_2,s_2\in[1,\infty]$ satisfy the conditions  
$$r\le s, \quad r_2,s_2\in[r,s],\quad\text{and}\quad
\frac{1}{r_1}-\frac{1}{r_2}=\frac{1}{s_1}-\frac{1}{s_2}=1-\frac{1}{r}-\frac{1}{s},$$
then for every symbol 
$a\in M^{r,s}_{\Wig(\phi_1,\phi_2)}(\pi^\#,\theta^A\times\theta^A)$ we have 
a bounded linear operator 
$$\Op^{\theta^A}(a)\colon M^{r_1,s_1}_{\phi_1}(\pi,\theta^A)\to M^{r_2,s_2}_{\phi_2}(\pi,\theta^A).$$ 
Moreover, the linear mapping 
$$\Op^{\theta^A}\colon  M^{r,s}_{\Wig(\phi_1,\phi_2)}(\pi^\#,\theta^A\times\theta^A)
\to\Bc(M^{r_1,s_1}_{\phi_1}(\pi,\theta^A),M^{r_2,s_2}_{\phi_2}(\pi,\theta^A))$$ 
is continuous. 
\end{corollary}

\begin{proof} 
It follows by Theorem~\ref{smooth_vect} that the space of smooth vectors for the representation $\pi$ is the Schwartz space $\Sc(G)$. 
Moreover, Corollary~\ref{conditions} shows that we can apply 
Corollary~\ref{C2} for the representation $\pi$. 
Now the conclusion follows by using the latter corollary. 
\end{proof}

\begin{corollary}\label{C3_mag}
Assume the setting of Corollary~\ref{conditions},  
let $\phi_1,\phi_2\in\Sc(G)\setminus\{0\}$, 
and $r,s\in[1,\infty]$ such that $\frac{1}{r}+\frac{1}{s}=1$.  
Then for every $a\in M^{r,s}_{\Wig(\phi_1,\phi_2)}(\pi^{\#},\theta^A\times\theta^A)$ we have 
$\Op^{\theta^A}(a)\in\Bc(L^2(G))$.   
Moreover, $\Op^{\theta^A}\colon M^{r,s}_{\Wig(\phi_1,\phi_2)}(\pi^\#,\theta^A\times\theta^A)\to\Bc(L^2(G))$ 
is a continuous linear mapping. 
\end{corollary}

\begin{proof}
This is the special case of Corollary~\ref{C2_mag} with with 
$r_1=s_1=r_2=s_2=2$, since  
Remark~\ref{mod_L2} shows that $M^{2,2}_{\phi_j}(\pi,\theta^A)=L^2(G)$ 
for $j=1,2$. 
\end{proof}

\begin{corollary}\label{C4_mag}
Assume the setting of Corollary~\ref{conditions} and  
let $\phi_1,\phi_2\in\Sc(G)\setminus\{0\}$.  
Then for every $a\in M^{1,1}_{\Wig(\phi_1,\phi_2)}(\pi^{\#},\theta^A\times\theta^A)$ we have 
$\Op^\theta(a)\in\Sg_1(L^2(G))$, and  
the linear mapping 
$\Op^{\theta^A}\colon M^{1,1}_{\Wig(\phi_1,\phi_2)}(\pi^\#,\theta^A\times\theta^A)\to\Sg_1(L^2(G))$ is continuous. 
\end{corollary}

\begin{proof}
Recall from Theorem~\ref{smooth_vect} that the space of smooth vectors for the representation $\pi$ is the Schwartz space $\Sc(G)$. 
Moreover, Corollary~\ref{conditions} shows that we can  
use Proposition~\ref{trace_class}, and the conclusion follows. 
\end{proof}

\begin{remark}\label{comparison}
\normalfont
In the special case when $G$ is the abelian group $(\RR^n,+)$ and we have the magnetic potential $A\in\Omega^1(\RR^n)$, 
the magnetic Weyl calculus 
$$\Op^{\theta^A}\colon\Sc'(\RR^n\times(\RR^n)^*)\to\Lc(\Sc(\RR^n),\Sc'(\RR^n))$$
is just the one constructed in \cite{MP04}. 
In this setting, we note the following: 
\begin{enumerate}
\item 
In the case when the coefficients of the magnetic field $B:=\de A\in\Omega^2(\RR^n)$ belong to the Fr\'echet space $\text{BC}^\infty(\RR^n)$ 
of smooth functions on $\RR^n$ which are bounded along with all of their partial derivatives, 
one established in \cite{IMP07} some 
sufficient conditions on a symbol $a\in\Sc'(\RR^n\times(\RR^n)^*)$ 
that ensure that the magnetic pseudo-differential operator $\Op^{\theta^A}(a)$ 
is bounded on $L^2(\RR^n)$. 
In this connection, we note that the previous Corollary~\ref{C3_mag} 
provides another type of sufficient conditions for $L^2$-boundedness 
when the coefficients of the magnetic field $B$ belong to the larger LF-space $\Cpol(\RR^n)$
of smooth functions on $\RR^n$ that grow polynomially together with 
their partial derivatives of arbitrary order. 
This follows since for every closed 2-form $B\in\Omega^2(\RR^n)$ 
whose coefficients belong to $\Cpol(\RR^n)$, 
one can construct in the usual way a 1-form $A\in\Omega^1(\RR^n)$ 
whose coefficients belong to $\Cpol(\RR^n)$ again such that $\de A=B$.
\item 
It follows by the comments preceding Corollary~\ref{C2_mag} that the modulation spaces of symbols $M^{r,s}_\Phi(\pi^\#,\theta^A\times\theta^A)$ 
can be alternatively described in terms of the modulation mapping 
which was introduced in \cite{MP09} in the case of the abelian group $G=(\RR^n,+)$ 
by using the magnetic Moyal product $\#^A$. 
It had been already noted in \cite{MP04} that the magnetic Moyal product 
on $(\RR^n,+)$ actually depends only on the magnetic field $B=\de A$. 
This assertion holds true for the two-step nilpotent Lie groups, as an easy consequence of the formula established in Th.~4.7 in \cite{BB09a}. 
\end{enumerate}
\qed
\end{remark}

\noindent\textbf{Acknowledgment.} 
The second-named author acknowledges partial financial support from the Project MTM2007-61446, DGI-FEDER, of the MCYT, Spain.

\end{document}